\documentclass{amsart}
\usepackage{luatex85}
\usepackage[normalem]{ulem}
\numberwithin{equation}{section}
\usepackage[all]{xy}
\usepackage{amssymb}
\usepackage{comment}
\usepackage{stmaryrd}
\usepackage{xcolor}
\usepackage{mathtools}
\definecolor{Vino}{rgb}{0.356,0,0}
\definecolor{Ruta}{rgb}{0.309, 0.459, 0.208}
\usepackage[colorlinks,linkcolor=Vino,citecolor=Ruta]{hyperref}
\let\cal\mathcal
\def\Ascr{{\cal A}}

\def\Cscr{{\cal C}}
\def\Dscr{{\cal D}}

\def\Fscr{{\cal F}}
\def\Gscr{{\cal G}}

\def\Iscr{{\cal I}}

\def\Lscr{{\cal L}}
\def\Mscr{{\cal M}}

\def\Oscr{{\cal O}}
\def\Pscr{{\cal P}}

\def\Rscr{{\cal R}}
\def\Sscr{{\cal S}}
\def\Tscr{{\cal T}}
\def\Uscr{{\cal U}}
\def\Vscr{{\cal V}}

\def\Xscr{{\cal X}}

\def\Zscr{{\cal Z}}
\let\blb\mathbb

\def \PP{{\blb P}}

\def \ZZ{{\blb Z}}

\def \NN{{\blb N}}
\def \RR{{\blb R}}

\def\id{\text{id}}

\def\Res{\operatorname{Res}}

\def\Lotimes{\overset{L}{\otimes}}
\def\quot{/\!\!/}

\def\Mod{\operatorname{Mod}}
\def\mod{\operatorname{mod}}
\def\Gr{\operatorname{Gr}}

\def\Supp{\mathop{\text{\upshape Supp}}}

\def\Qch{\operatorname{Qch}}
\def\coh{\mathop{\text{\upshape{coh}}}}

\def\Spec{\operatorname {Spec}}

\def\Hom{\operatorname {Hom}}
\def\uHom{\operatorname {\mathcal{H}\mathit{om}}}
\def\uEnd{\operatorname {\mathcal{E}\mathit{nd}}}
\def\End{\operatorname {End}}

\def\RHom{\operatorname {RHom}}
\def\uRHom{\operatorname {R\mathcal{H}\mathit{om}}}

\def\ker{\operatorname {ker}}

\def\End{\operatorname {End}}

\def\id{{\operatorname {id}}}

\def\rk{\operatorname {rk}}

\def\gldim{\operatorname {gl\,dim}}
\def\G{\mathop{\underline{\underline{{\Gamma}}}}\nolimits}
\def\r{\rightarrow}
\def\u{\uparrow}

\DeclareMathOperator{\Proj}{Proj}

\DeclareMathOperator{\Ind}{Ind}

\DeclareMathOperator{\Aut}{Aut}

\def\Bl{\rm{Bl}}
\def\codim{\operatorname{codim}}

\def\uRHom{\operatorname {R\mathcal{H}\mathit{om}}}
\def\uREnd{\operatorname {R\mathcal{E}\mathit{nd}}}

\def\lll{{\operatorname{loc}}}
\def\G{G}

\newtheorem{lemma}{Lemma}[section]
\newtheorem{proposition}[lemma]{Proposition}
\newtheorem{theorem}[lemma]{Theorem}
\newtheorem{corollary}[lemma]{Corollary}

\theoremstyle{definition}

\newtheorem{example}[lemma]{Example}
\newtheorem{definition}[lemma]{Definition}

{

\newtheorem{step}{Step}

}

\theoremstyle{remark}

\newtheorem{remark}[lemma]{Remark}

\newdimen\uboxsep \uboxsep=1ex
\def\uboxn#1{\vtop to 0pt{\hrule height 0pt depth 0pt\vskip\uboxsep
\hbox to 0pt{\hss #1\hss}\vss}}

\def\uboxs#1{\vbox to 0pt{\vss\hbox to 0pt{\hss #1\hss}
\vskip\uboxsep\hrule height 0pt depth 0pt}}

\def\Sym{\operatorname{Sym}}

\let\oldmarginpar\marginpar
\def\marginpar#1{\oldmarginpar{\raggedright \tiny \baselineskip 0pt \lineskip 0pt #1}}

\def\Bl{\operatorname{Bl}}
\def\Sym{\operatorname{Sym}}

\def\rep{\operatorname{rep}}
\def\fks{{\bf X}}
\def\K{R}

\def\Sym{\operatorname{Sym}}
\def\Ob{\operatorname{Ob}}

\def\ns{{ns}}

\def\XZ{Z}
\def\XscrZ{\Zscr}
\def\XZH{Z^{\langle H\rangle}}
\def\GpH{H}
\def\ns{ns}
\def\u{u}
\title[]{Comparing the Kirwan and noncommutative resolutions of quotient varieties}
\author[\v{S}pela \v{S}penko and Michel Van den Bergh]{\v{S}pela \v{S}penko and Michel
  Van den Bergh} 
\thanks{While working on this project the first author was a FWO $[$PEGASUS$]^2$
  Marie Sk\l odowska-Curie fellow at the Free University of Brussels
  (funded by the European Union Horizon 2020 research and innovation
  program under the Marie Sk\l odowska-Curie grant agreement No
  665501 with the Research Foundation Flanders (FWO)).}  
\address[\v{S}pela \v{S}penko]{D\'epartement de Math\'ematique, Universit\'e Libre de Bruxelles, Campus de la Plaine CP 213, Bld du Triomphe, B-1050 Bruxelles}
\email{spela.spenko@ulb.be}
\address[Michel Van den Bergh]{Vakgroep Wiskunde, Universiteit Hasselt, Universitaire Campus \\
  B-3590 Diepenbeek} 
\email{michel.vandenbergh@uhasselt.be}
\address[Michel Van den Bergh]{VUB Main Campus Etterbeek,
Pleinlaan 2,
B-1050 Elsene} 
\email{michel.van.den.bergh@vub.be}
\thanks{The second author is a senior researcher at the Research
  Foundation Flanders (FWO).  This project has received funding from the European Research Council (ERC) under the European Union's Horizon 2020 research and innovation programme (grant agreement No 885203). This project was also supported the FWO grant G0D8616N: ``Hochschild cohomology and
  deformation theory of triangulated categories''.}
\begin{document}

\begin{abstract}
  Let a reductive group $G$ act on a smooth variety $X$ such that a good quotient $X\quot G$ exists. We show that
  the derived category of a noncommutative crepant resolution (NCCR) of $X\quot G$, obtained from a $G$-equivariant vector bundle on $X$, 
can be embedded in the 
  derived category of the (canonical, stacky) Kirwan resolution of $X\quot G$. 
 In fact 
the embedding 
  can be completed to a semi-orthogonal decomposition in which the other parts
are all derived categories of Azumaya algebras over smooth Deligne-Mumford stacks.
\end{abstract}

\maketitle
\tableofcontents

\section{Introduction}
\subsection{Preliminaries}
We fix an algebraically closed ground field $k$ of characteristic~$0$. Everything is linear over $k$.
Here and below, if $\Xscr$ is an
Artin stack and $\Lambda$ is a quasi-coherent sheaf of rings on~$\Xscr$ then 
$D(\Lambda)$ is the unbounded derived categories of right $\Lambda$-modules with quasi-coherent cohomology. We also put $D(\Xscr):=D(\Oscr_\Xscr)$.

\medskip

We recall the following definition.
\begin{definition}\cite{VdB04}
Let $R$ be a normal Gorenstein 
 domain. A \emph{noncommutative crepant resolution} (NCCR) of $R$
is an $R$-algebra of finite global dimension 
   of the form $\Lambda=\End_R(M)$  which in addition is Cohen-Macaulay as $R$-module and where $M$ is a non-zero finitely generated 
reflexive
$R$-module.
\end{definition}
In this paper we will say
that a sheaf of $k$-algebras $\Lambda$ on a scheme $X$ is a NCCR of $X$ if $\Lambda(U)$ is a NCCR of $\Gamma(U)$ 
for every connected affine open $U\subset X$.

\medskip

The derived categories of NCCRs are particular instances of  ``categorical
strongly crepant resolutions'' and the latter are conjectured to be minimal among all ``categorical
resolutions'' \cite{Kuz}. 
In the current paper we provide new evidence for this conjecture. Namely
 we will show that the NCCRs of quotient
singularities for reductive groups, of the type constructed in \cite{SVdB},
embed
in a particular canonical (stacky)  resolution of singularities, constructed by Kirwan in \cite{Kirwan1}.

\begin{remark}
\label{rem:care}
  The correct interpretation of the conjecture requires some care since for
  example if $X$ is a noetherian scheme and $\pi:Y\r X$ is a
  commutative resolution of singularities (where $Y$ can be a smooth Deligne-Mumford stack) 
  then $D(Y)$ is only a categorical resolution of $D(X)$ if $X$ has
  rational singularities \cite[Example 5.1]{Lunts}.\footnote{On the other hand
  in~\cite{KuznetsovLunts} it is shown that an arbitrary commutative
  resolution can always be suitably modified to yield a
  categorical resolution.}
\end{remark}

\medskip

To be able to state our main results we introduce some more
notation. Let $G$ be a reductive group and let $X$ be a smooth
$G$-variety such that a good quotient\footnote{A  good quotient of $X$ is a map $X\r Y$ which is locally (on $Y$) of the form $\Spec R\r \Spec R^G$. If $Y$ exists
then it is unique and we write $X\quot G\coloneqq Y$.}  $\pi:X\to X\quot G$ exists.
In \cite{Kirwan1}, Kirwan constructed (for projective $X$) a partial resolution of $X\quot G$
by an inductive procedure involving GIT quotients 
of  repeated $G$-equivariant blowups of $X$ 
(see \S\ref{sec:embedding}).  
The final quotient variety $\fks\quot G$ is then a partial resolution of singularities of $X\quot G$ (finite quotient singularities may remain).
We may also view the end result as a smooth Deligne Mumford stack $\fks/G$ and therefore we say that $\fks/G$ is the {\em{Kirwan (stacky) resolution}} of $X\quot G$.
In \cite{EdidinRydh},  Edidin and Rydh generalised the Kirwan (and also Reichstein \cite{Reichstein}) procedure to irreducible Artin stacks with stable good moduli spaces. We will heavily use their technical results.

\subsection{Assumptions}

Let $X^{\ns}\subset X$ be the locus of points whose stabilizer is not finite or whose orbit is not closed (see \S\ref{sec:general}). Throughout the introduction (and in various parts of the paper) we assume 
\begin{equation}\tag{H2}\label{H2}
\codim(X^{\ns},X)\geq 2.
\end{equation}
Occasionally we will impose the slightly stronger condition that $X$ is \hyperlink{gnc}{\em{generic}} in the sense of \cite{SVdB}; i.e.\ that $G$ acts in addition freely on an open subset of $X-X^{\ns}$ whose
complement has codimension $\ge 2$
(see \S\ref{sec:hhgencodim1}). 
\subsection{The embedding of a noncommutative resolution in $D(\fks /G)$}
\label{sec:embedding1}

In this paper we consider noncommutative resolutions of $X\quot G$ of the form
\begin{equation}
\label{eq:resU}
\Lambda=\uEnd_X(\Uscr)^G
\end{equation}
where $\Uscr$ is a $G$-equivariant vector bundle on $X$.  This is a minor generalization
with respect to \cite{SVdB} where we exclusively considered the case
  $\Uscr=U\otimes \Oscr_X$ where~$U$ is a finite dimensional
  representation of $G$.\footnote{In the context of the current paper the 
  restriction $\Uscr=U\otimes \Oscr_X$ is unnatural as we will, in any case, forcibly encounter
  more general equivariant vector bundles when iterating Reichstein transforms.}

\medskip

A feature of a resolution like \eqref{eq:resU} is that there is an embedding
\[
- \Lotimes_{\Lambda}\Uscr: D(\Lambda)\hookrightarrow D(X/G).
\]
Now let $\fks/G$ be the Kirwan resolution of $X\quot G$ which we factor as
\[
\fks/G\xrightarrow{\Xi} X/G\r X\quot G.
\]
As $X\quot G$ has rational singularities \cite{Boutot}, $D(\fks/G)$ is a categorical resolution of $D(X\quot G)$ by Remark \ref{rem:care},
which implies in particular that pullback provides an embedding
\[
D(X\quot G)\hookrightarrow D(\fks/G).
\]
It is important to observe however that \emph{we do not have an embedding of $D(X/G)$ in $D(\fks/G)$}; an indication for this is given in \S\ref{sec:conifold}
where we provide an example with $\rk K_0(X/G)=\infty$ but $\rk K_0(\fks/G)=8$.

\medskip

 The following is our first main result.
\begin{proposition}[Proposition \ref{prop:embedding}]\label{prop1}
  Let $G$ be a reductive group acting on a smooth variety $X$ such
  that a good quotient $\pi:X\to X\quot G$ exists. Assume \eqref{H2}. 
Let $\fks/G$ be the
  Kirwan resolution of $X\quot G$ discussed above.

\medskip

 Let $\Uscr$ be a $G$-equivariant
  vector bundle on $X$ and assume that
  $\Lambda=\uEnd_X(\Uscr)^G$ is (locally) Cohen-Macaulay as a sheaf
  of algebras on $X\quot G$ (e.g. if $\Lambda$ is an NCCR). Then the composition
\[
D(\Lambda)\hookrightarrow D(X/G)\xrightarrow{L\Xi^\ast} D(\fks/G)
\]
is fully faithful.
\end{proposition}

\subsection{The Reichstein transform and a naive embedding}
One establishes Proposition \ref{prop1} inductively, following the 
Edidin and Rydh procedure discussed above. Let us describe the procedure more precisely. 

For a point $x\in X$ let $G_x$ be its stabilizer 
and set $\mu(X)=\max_{x\in X}\dim G_x$.
Put
\begin{align*}
Z&=\{x\in X\mid \dim G_x=\mu(X)\},\\
\bar{Z}&=\{x\in X\mid \overline{Gx}\cap Z\neq 0\}.
\end{align*}
Both $Z$ and $\bar{Z}$ are closed in $X$ (see \S\ref{subsec:Kstep}). Denote by $(-)'$ the strict transform of a closed subset under a blowup. Put 
\[
X^\K=\Bl_Z X-\bar{Z}'.
\]
The resulting map $\xi^R:X^\K/G\to X/G$ is called the {\emph{Reichstein transform}} of $X/G$. 
One has $\mu(X^\K)<\mu(X)$ and hence by performing a sequence of such transforms the maximal stabilizer dimension becomes zero, yielding $\fks$ and  the Kirwan resolution $\fks /G$. 

Let $\Uscr'=\xi^{R\ast} \Uscr$ be a vector bundle on $X^R/G$ and let 
$\Lambda'=\uEnd(\Uscr')^G$, viewed as a sheaf of algebras on $X^R\quot G$. 

We first obtain a naive embedding. 

\begin{proposition}[Corollary 5.3, Lemma 5.4]\label{prop:naive}
Assume \eqref{H2}. If $\Lambda$ is Cohen-Macaulay, then so is $\Lambda'$. Moreover,
 pullback for the morphism of ringed spaces $(X^R\quot G,\Lambda')\r  (X\quot G,\Lambda)$ induces an
 embedding of derived categories $D(\Lambda) \hookrightarrow D(\Lambda')$.
\end{proposition}

 We obtain Proposition \ref{prop1} by successive application of Proposition \ref{prop:naive}.

\subsection{Semi-orthogonal decomposition of $D(\fks/G)$}
\label{sec:sodintro}
In the case that $\Lambda$ in the statement of Proposition \ref{prop1}
is actually a noncommutative crepant resolution of $X$ (and $X$ is generic), another main result of this paper is that the embedding
$D(\Lambda)\hookrightarrow D(\fks/G)$ can be completed to a semi-orthogonal decomposition.  

\medskip

Here we first encounter an impediment, wishing to   proceed in the inductive way using the ``naive'' embedding given by Proposition \ref{prop:naive}. 
The hindrance is that the NCCR property is not preserved when passing from $\Lambda$ to $\Lambda'$ (see Example \ref{sec:example}). 
Therefore we would not be able to proceed inductively, even if we could enhance the embedding in Proposition \ref{prop:naive} to a semi-orthogonal decomposition. 

\medskip

The most serious issue is that finite global dimension is not preserved. 
This obstacle we overcome  by slightly tweaking $\Uscr'$, and hence $\Lambda'$. 
Let $\Oscr_{\Bl_Z X}(1)$ be the tautological relatively ample line bundle on $\Bl_ZX$ and
let $\Oscr_{X^{\K}}(1)$ be its restriction to $X^{\K}$. For some $N>0$, $\Oscr_{X^R}(N)$ is the pullback of a line bundle $(\pi^R_*\Oscr_{X^R}(N))^G$ on the quotient $\pi^R:X^R\to X^R\quot G$ (see Proposition \ref{prop:Reichsteinproperties}\eqref{5}).  
From a vector bundle $\Uscr$ on $X/G$ we produce the vector bundle $\Uscr^R$ on $X^R/G$ as
\[
{\Uscr^R}=
\bigoplus_{i=0}^{N-1} (\xi^{R\ast} \Uscr)(i).
\]

We obtain an {\em Orlov's type (blow-up) semi-orthogonal decomposition} for $\Lambda^R=\uEnd_{X^R}(\Uscr^R)$ with one component corresponding to $\Lambda$ 
and the other components corresponding to representatives $Z_i$ for the orbits of the $G$-action on the connected components of the center $Z$ of the blow-up.
Let $G_i\subset G$ be the stabilizer of $Z_i$, as a connected component.

\begin{proposition}[Corollary \ref{cor:sod}]\label{prop:introsod}
Assume \eqref{H2} and that $\Lambda$ is Cohen-Macaulay. Let $\Uscr_{Z_i}$ be the restriction of $\Uscr$ to $Z_i$ and let $\Lambda_{Z_i}=\uEnd_{Z_i}(\Uscr_{Z_i})^{G_i}$. 
There is a semi-orthogonal decomposition 
\[
D(\Lambda^{\K})\cong \langle D(\Lambda),\underbrace{D(\Lambda_{Z_1}),\ldots,  D(\Lambda_{Z_1})}_{c_1-1},\ldots,
\underbrace{ D(\Lambda_{Z_t}),\ldots,  D(\Lambda_{Z_t})}_{c_t-1}\rangle
\]
where   $c_i=\codim(Z_i,X)$. 
Moreover, the components corresponding to different $Z_i$ are orthogonal.
\end{proposition}
Unfortunately it turns out that 
the NCCR
property is \emph{still} not preserved by the passage $\Lambda\mapsto \Lambda^R$; the culprit being that the Reichstein transform may produce non-trivial
stabilizers in codimension one.\footnote{The appearance of non-trivial stabilisers in codimension one  is already visible for an action of $G_m$ with only odd weights. Then the exceptional divisor is stabilised by $\ZZ_2$. In that case, however, the induction step is necessary. The same phenomenon can occur for a higher dimensional example, e.g. $G_m^2$ with one of the components of $G_m$ acting with odd weights.} We solve this by introducing the
following two technical conditions.
\begin{itemize}
\item[($\alpha$)] $\Lambda$ is \emph{homologically homogeneous}  (see Definition \ref{def:hh}).
\item[($\beta$)] $\Uscr$ is a \emph{generator in codimension one} (see Definition \ref{def:genincodim1}).
\end{itemize}
Both of these conditions are satisfied if $X$ is generic and $\Lambda$ is a NCCR (see Proposition \ref{prop:NCCRimpliesU1}). Moreover we prove that both properties, along with the \eqref{H2} property,
 are preserved
under the passage $X\mapsto X^R$ (see Propositions \ref{prop:Reichsteinproperties},\ref{prop:want}).

\medskip

The successive applications of  semi-orthogonal decompositions as in Proposition \ref{prop:introsod} following successive Reichstein transforms  yield a semi-orthogonal decomposition of 
$D(\bold{\Lambda})$ where $\bold{\Lambda}$ is obtained on the final step.
We will show that if $(\alpha,\beta)$ hold for $\bold{\Lambda}$ then in fact
 $D(\bold{\Lambda})\cong D(\fks/ G)$. 
So by the above discussion we conclude that it is enough to 
assume $(\alpha,\beta)$ hold for the initial $\Lambda$ to obtain a semi-orthogonal decomposition of $D(\fks/G)$. This semi-orthogonal decomposition is stated in Theorem \ref{thm:sod}. We will not restate it here
as we prefer to give a more geometric version in the next section.

\subsection{Geometric description of the semi-orthogonal decomposition}
We further proceed to give a geometric description of the $D(\Lambda_{Z_i})$ appearing in Proposition~\ref{prop:introsod}. 
For simplicity we here state our final result only  in the abelian case. For the general case see Theorem  \ref{thm:sod}, Corollary \ref{cor:external}.

Let us assume  the Kirwan resolution is obtained by performing~$n$ successive Reichstein transforms with $Z_j$ being blown up at the $j$-th step. Let $Z_{ji}$, $1\leq i\leq t_j$, be representatives for the orbits of the $G$-action on the connected components of $Z_j$. Let $H_{ji}$ be the  stabilizer of $Z_{ji}$.
\begin{theorem}[Theorem \ref{thm:sod}, Remark \ref{rem:spoiler}, Corollary \ref{cor:external}]\label{thm:introGsod}
Assume \eqref{H2}. Assume that $\Lambda$ is homologically homogeneous and that $\Uscr$ is a generator in codimension $1$. Let $G$ be abelian (for general $G$ see Theorem  \ref{thm:sod}, Corollary \ref{cor:external}). There is a semi-orthogonal decomposition 
\[
D(\fks/G)\cong \langle D(\Lambda), D(Z_{ji}/(G/H_{ji}))^{\oplus N_{ji}}_{1\leq j\leq n, 1\leq i\leq t_j,0\leq k\leq c_{ji}-2}\rangle
\]
for some $N_{ji}\in \NN_{>0}$, 
where $c_{ji}:=\codim(Z_{ji},X_j)$, and the terms appear in lexicographic order (according to the label $(j,i,k)$).
\end{theorem}
As we have already mentioned in \S\ref{sec:sodintro},  by Proposition \ref{prop:NCCRimpliesU1} the conditions for this theorem are satisfied if $X$ is generic  and $\Lambda$ is a NCCR of $X\quot G$.

\medskip

For general $G$, $Z_{ji}$ will not have a global stabilizer group, however the generic stabilizer is conjugate to a fixed group $H_{ji}$. Thus, instead of  $Z_{ji}/(G/H_{ji})$ we should take $Z_{ji}^{\langle H_{ji}\rangle}/(N_V(H_{ji})/H_{ji})$ where $Z_{ji}^{\langle H_{ji}\rangle}$ is a suitable (smooth) subscheme of $Z_{ji}^{H_{ji}}$ and $N_V(H_{ji})$ is a subgroup of the normalizer group $N(H_{ji})$, and adorn it with a sheaf of (equivariant) Azumaya algebras (see Corollary \ref{cor:external}). 
\section{Acknowledgement}
A significant part of this work was done during a research in pairs program at the ``Centro Internazionale per la Ricerca Matematica'' in Trento. 
Further work was carried out at the ``Max-Planck-Institut f\"ur Mathematik'' in Bonn.
The authors thank both institutes for the excellent working conditions and invigorating atmosphere. 
They are grateful to Michel Brion for discussions around Proposition~\ref{prop:structure}.
They also thank Geoffrey Janssens for helpful and constructive discussions. 
Moreover, they thank Amnon Yekutieli for several useful comments, in particular those related to the appendix. Furthermore, they thank the referee for very helpful suggestions.

\section{Notation and conventions}\label{sec:notation}
We fix  an algebraically closed field $k$ of characteristic $0$.  Everything is linear over $k$. 
In particular $\Spec k$ is the base scheme and unadorned tensor products are over $k$. 

All schemes are separated. The stacks we will use are global quotients
stacks $X/G$ for which $X$ is at least separated.  We will
silently identify $G$-equivariant sheaves on $X$ and sheaves on
$X/G$. If a good quotient $\pi:X\to X\quot G$ exists we write
$\pi_s:X/G \to X\quot G$ for the corresponding stack morphism. On some occasions
we sloppily write $(-)^G$ for $\pi_{s\ast}$. We sometimes silently globalize results for $X$, $X/G$, $X\quot G$, \dots which are available in the literature for $X$ affine
and which are seen to be trivially local over $X\quot G$.

All modules are right modules. If $\Lambda$ is ring then $D(\Lambda)$ is the unbounded derived category of $\Lambda$. 
If $\Xscr$ is an
Artin stack and $\Lambda$ is a quasi-coherent sheaf of rings on $\Xscr$ then 
$D(\Lambda)$ is the unbounded derived categories of $\Lambda$-modules with quasi-coherent cohomology. We also put $D(\Xscr):=D(\Oscr_\Xscr)$.

For an affine  algebraic group $H$ we denote by $H_e$ the identity component of $H$. 
and we let $\rep(H)$ be the set of isomorphism classes of irreducible $H$-representations.

For $U,U'\in \Ob(\mathfrak{a})$ where $\mathfrak{a}$ is a Karoubian category
we write $U:=:U'$ to indicate $U\in \Ob(\operatorname{add}(U'))$ and $U'\in \Ob(\operatorname{add}(U))$.

Unless otherwise specified, ``graded'' means $\ZZ$-graded and elements of a graded ring are automatically
assumed to be homogeneous.
\section{Generalities}\label{sec:general}

Unless otherwise specified $X$ is a 
smooth  variety 
 and  $G$ is a reductive group acting on $X$ such
that a good quotient $\pi:X\r X\quot G$ exists (see e.g. \cite[Definition 3.3.1]{SVdB3} for the definition of good quotient).

\subsection{(Semi-)stability}

A point in $X$ is \emph{stable} if it has closed orbit and finite stabilizer. 
We write $X^s$ for the stable locus of $X$ and $X^{\ns}$ for its complement. 
If $\Lscr$ is a line bundle on $X$ which linearises the $G$-action then by \cite[\S4]{Mumford} $x\in X$ is ($\Lscr$-)semi-stable if there is  $f\in H^0(X,\Lscr^{\otimes n})^G$ for $n>0$ such that $f(x)\neq 0$ and $X_f$ is affine. We denote the set of $\Lscr$-semi-stable points by $X^{ss,\Lscr}$.

\begin{remark}\label{lem:torsion}
By \cite[Theorem 1.10]{Mumford}, a good quotient $\pi:X^{ss,\Lscr}\to X^{ss,\Lscr}\quot G$ exists. Moreover,  there is an $N>0$ such that the restriction of $\Lscr^{\otimes N}$ to $X^{ss,\Lscr}$ is the pullback of a line bundle on $X^{ss,\Lscr}\quot G$.  
 It follows in particular that any $\Lscr$-semi-stable point 
$x$ has  a $G$-equivariant saturated\footnote{A $G$-invariant open subset in $X$ is saturated if it is a pullback of an open subset in $X\quot G$.}  neighbourhood on which $\Lscr$ is torsion. 
\end{remark}

A particular example of a linearisation is
given by a line bundle of the form $\Lscr=\chi\otimes \Oscr_X$ for
$\chi\in X(G)$. We can sometimes reduce to this case by Lemma \ref{rem:vecbundle} below.

\subsection{(Semi-)stability and \'etale maps}\label{subsec:semisetale}

\begin{lemma}\label{lem:etalescss}
Assume that $\phi:Y\to X$ is a $G$-equivariant \'etale map. 
Let $x\in X$ and let $y\in Y$ be a preimage of $x$. Then the following holds true: 
\begin{enumerate}
\item 
$G_y\subset G_x$ and $\dim G_y=\dim G_x$.
\item 
If $Gx$ is closed then so is $Gy$.
\item
If $x$ is stable then so is $y$.
\end{enumerate}
In addition, if $\phi$ is strongly \'etale\footnote{A $G$-equivariant map $\phi:Y\to X$ is strongly \'etale if it is induced by pullback from an \'etale map $V\to X\quot G$. In particular the inclusion of a saturated open subset is strongly \'etale.} then $G_x=G_y$ and the converse of  (2) and (3) holds. 
\end{lemma}

\begin{proof}
(1) is clear since $\phi$ is quasi-finite. For (2) 
assume that $Gx$ is closed and  $Gy$ is not closed. Since the action of $G$ on $Gx$ is transitive and $\overline{Gy}\subset \phi^{-1}(Gx)$, we have $\phi(\overline{Gy}\setminus Gy)=Gx$. Hence $\dim(\overline{Gy}\setminus Gy)=\dim Gx=\dim Gy$ (as $\phi$ is quasi-finite), which is a contradiction. (3) follows by combining (1) and (2). 

Now assume $\phi$ is strongly \'etale. By definition, $Y=V\times_{X\quot G} X$ with $V\to X\quot G$ \'etale. Let $\bar{y}$, $\bar{x}$ be the images of $y$, $x$ in $V$, $X\quot G$, respectively, and let $Y_{\bar y}$, $X_{\bar x}$ be the corresponding fibers. Then $\phi$ induces an isomorphism $Y_{\bar y}\cong X_{\bar x}$ and hence $G_x=G_y$. If $Gy$ is closed then it is closed in $Y_{\bar y}$, and hence $Gx$ is closed in $X_{\bar x}$, and therefore closed in $X$. This proves the converse of (2). The converse of (3) is again a combination of (1) and the converse of (2). 
\end{proof}

\begin{lemma}
Assume that $\phi:Y\to X$ is a $G$-equivariant \'etale map which is moreover affine. 
Let $x\in X$ and let $y\in Y$ be a preimage of $x$. Assume $\Lscr$ is a linearisation of the $G$-action on $X$ and let $\Mscr=\phi^*\Lscr$. 
If $x$ is $\Lscr$-semi-stable then $y$ is $\Mscr$-semi-stable. If $\phi$ is strongly \'etale and $Y$ and $X$ are affine then the converse also holds.
\end{lemma}

\begin{proof}
The first part follows by pulling back the section nonvanishing on $x$ to $Y$. 
The converse follows by considering the restriction of $L$ and $M=\underline{\Spec}\Sym\Mscr$ to $X_{\bar{x}}$ and $Y_{\bar{y}}$, respectively (with notation as in the proof of Lemma \ref{lem:etalescss}). 
\end{proof}

\subsection{Genericity conditions}
Let $i\in \NN$. Below we write $(Hi)$ for the condition $\codim(X^{\ns},X)\ge i$. 
Note that $(H1)$ is equivalent to $X^s\neq \emptyset$.
Furthermore $(Hi)\Rightarrow (Hj)$ for $j\le i$.

\subsection{Reduction to the linear case}\label{subsec:prelimlin}
We will often reduce to the linear case using the Luna slice theorem \cite{Luna}. 
Assume that $x$ is a point in $X$ with closed orbit. There there is a
smooth affine slice $S$ at $x$ such that there is a strongly \'etale
map $\phi:G\times^{G_x}S\to X$. Furthermore we may assume that there is a
strongly \'etale map $\gamma:S\to T_xS$, sending $x$ to $0$. We
will usually abuse terminology by simply calling $S$  a slice as $x$ and by calling $(G_x,T_xS)$ the \emph{linearised slice} at
$x$.

\begin{lemma}\label{lem:H12slices}
The hypothesis $(Hi)$ holds for $(G,X)$ if and only it holds for $(G_x,T_xS)$ for all points $x\in X$ with  closed orbit. 
\end{lemma}
\begin{proof}
  Let $x$ be a point in $X$ with closed orbit. We first show that $\codim (X^{\ns},X)\le \codim
((T_xS)^{\ns},T_xS)$, so that if $(Hi)$ holds for $(G,X)$ it holds for $(G_x,T_xS)$.

If $x\not\in X^{\ns}$ then $G_x$ is finite
and hence $(T_xS)^{\ns}$ is empty so that there is nothing to prove.

Now assume $x\in X^{\ns}$  and let
$\phi:G\times^{G_x}S\to X$, $\gamma:S\r T_x S$ be as above. By Lemma \ref{lem:etalescss}, 
$\phi^{-1}(X^{\ns})=(G\times^{G_x}S)^{\ns}$.
In addition one can verify that $(G\times^{G_x} S)^{\ns}=G\times^{G_x} S^{\ns}$. We deduce
$\codim (X^{\ns},X)\le \codim(G\times^{G_x} S,(G\times^{G_x} S)^{\ns}))= \codim(S^{\ns},S)$ (as $\phi$ is \'etale). Let $\codim_x(S^{\ns},S):=\dim S-\dim\Spec \Oscr_{S^{\ns},x}$ be
the local codimension at $x$. Then we compute
\begin{align*}
\codim(S^{\ns},S)&\le \codim_x(S^{\ns},S)\\
&=\codim_0((T_x S)^{\ns},T_x S)\\
&=\codim((T_x S)^{\ns},T_x S).
\end{align*}
For the first equality we use that $\gamma$ is strongly \'etale and hence a local homeomorphism, and moreover $S^{\ns}=\gamma^{-1}((T_x S)^{\ns})$ by Lemma \ref{lem:etalescss}. 
For the second equality we use that $T_x S$ is a $G_x$-representation and that $(T_x S)^{\ns}$ is defined by a homogeneous ideal.

\medskip

To prove the converse we have to show that  $\codim (X^{\ns},X)=\codim
((T_xS)^{\ns},T_xS)$ for at least one $x$. By reversing the above arguments it follows that we may take $x$ to be a point with closed orbit in an irreducible component of 
$X^{\ns}$ of maximal dimension (guaranteeing $\codim(S^{\ns},S)=\codim_x(S^{\ns},S)$).
\end{proof}

\subsection{Equivariant vector bundles}
\label{sec:equivariant}

\begin{lemma}\label{lem:bundle}
  If $\Vscr$ is a $G$-equivariant vector bundle on $X$ and $x\in X$ is
  $G$-invariant point 
   then $x$ has a saturated affine $G$-invariant 
  neighborhood on which $\Vscr$ is of the form $V\otimes \Oscr_X$ for
  the $G$-representation $V$ which is the fiber of $\Vscr$ in $x$, i.e.\
$V=\Vscr\otimes_X k(x)$.
\end{lemma}

\begin{proof}
  By taking the pullback of an affine neighborhood of the image of $x$
  in $X\quot G$ we may reduce to the case that $X$ is affine.  Choose
  a $G$-invariant splitting $\Gamma(X,\Vscr)\to V$
  (since $X$ is affine $V=\Gamma(X,\Vscr) \otimes k(x)$). 
This gives us an $G$-equivariant map
  $V \otimes \Oscr_X\to \Vscr$
 which is an isomorphism in a
  neighborhood of $x$. The maximal neighborhood $U$ on which  this
  is the case must be $G$-equivariant and open. Then $\pi(X\setminus U)=(X\setminus U)\quot G$ and
$\{\pi(x)\}=\{x\}\quot G$ are disjoint closed subsets of $X\quot G$ (see \cite[Theorem
  1.24(iv)]{Brion}). Finally the saturated neigbourhood of $x$ we
  want is $X\setminus \pi^{-1}(\pi(X\setminus U))=
  \pi^{-1}(X\quot G\setminus \pi(X\setminus U))$; i.e. the maximal saturated
subset of $U$.
\end{proof} 

\begin{lemma}\label{rem:vecbundle}
  Let $\Vscr$ be a $G$-equivariant vector
  bundle on $X$. We may choose the slice $S$ as in \S\ref{subsec:prelimlin} in such a way that the pullback
  of $\Vscr$ to $G\times^{G_x}S$ coincides with
  $G\times^{G_x}(V\otimes \Oscr_S)$ for $V=\Vscr\otimes k(x)$. 
\end{lemma}
\begin{proof} 
First take an arbitrary slice $S$ at $x$. We pull back $\Vscr$ to $G\times^{G_x} S$ and replace
$X$ by $G\times^{G_x} S$.
 Now the $G$-equivariant vector bundle $\Vscr$ on $G\times^{G_x}S$ restricts to a
  $G_x$-equivariant vector bundle $\Vscr_S$ on $S$, such that $\Vscr=G\times^{G_x} \Vscr_S$.
We then apply Lemma
  \ref{lem:bundle} to obtain a saturated affine
  $G_x$-invariant open subset $x\in S'\subset S$ such that $\Vscr_S\mid S'\cong V\otimes \Oscr_{S'}$. Since a saturated open
  immersion is a special case of a strongly \'etale map,
  $G\times^{G_x} S' \hookrightarrow G\times^{G_x} S\to X$ is strongly
  \'etale and so we may replace $S$ by $S'$.
\end{proof}

\subsection{The canonical sheaf on $X^{s}/G$}\label{subsec:canonicalsheaf}
The following lemma gives the precise relation between the canonical sheaf of the stack $X^{s}/G$ and $\omega_{X^s}$ considered as a $G$-equivariant sheaf.  

\begin{lemma}\label{lem:omega}
There is an isomorphism $\omega_{X^{{s}}/G}\cong \wedge^{\dim G}\mathfrak{g}\otimes \omega_{X^{s}}$.
\end{lemma}

\begin{proof}
The lemma follows from the fact that the cotangent complex on $X/G$ is given by the complex $\Omega_{X}\to \mathfrak{g}^*\otimes \Oscr_{X}$ with $\Omega_X$ in degree zero. 
Since the map is surjective on $X^{ s}$ by the definition of $X^{ s}$, we get the exact sequence $0\to \Omega_{X^{ s}/G}\to \Omega_{X^{s}}\to \mathfrak{g}^*\otimes \Oscr_{X^{{s}}}\to 0$. Taking determinants we get the desired equality. 
\end{proof}

\begin{remark}
Note that $\alpha=\wedge^{\dim G}\mathfrak{g}$ can only be nontrivial  in the case when $G$ is disconnected. Furthermore $\alpha^2$ is always trivial (see e.g. \cite[p.41]{Knop2}).
\end{remark}
\section{The Kirwan resolution}\label{sec:setting}
Let $X$ be as in \S\ref{sec:general}. We assume in addition that $X$ satisfies (H1), i.e.\ $X^s\neq\emptyset$.
\subsection{The Reichstein transform}\label{subsec:Kstep}
The steps in the partial resolution of $X\quot G$ described in \cite{Kirwan1} were
reinterpreted by Reichstein \cite{Reichstein}, and generalized by Edidin and More \cite{EdidinMore} and Edidin and
Rydh \cite{EdidinRydh}. They are now known as ``Reichstein transforms''
\cite{EdidinMore}. 
We will use $(-)^{\K}$ for notations related to the Reichstein transform.

\medskip

Let $\mu$ be the maximal dimension of the stabilizers of the $G$-action on $X$ and 
for simplicity we put 
$
Z=X_\mu:=\{x\in X\mid \dim G_x=\mu\},
$ 
which is closed and smooth (see e.g. \cite[Proposition B.2]{EdidinRydh}). 
Assume $\mu>0$. 
Put
\[
\bar{Z}=\{x\in X\mid \overline{Gx}\cap Z\neq 0\}.
\]
Then $\bar{Z}=\pi^{-1}(Z\quot G)$, so it is closed as well.
Let $\xi:\tilde{X}\to X$ be the blowup of $X$ in $Z$ and let $\tilde{Z}\subset \tilde{X}$ be the strict transform
of $\bar{Z}$. Let $\xi^{\K}$ be the restriction of $\xi$ to ${X}^{\K}:=\tilde{X}-\tilde{Z}$. The resulting map $\xi^\K_s:X^\K/G\to X/G$ is called the {\emph{Reichstein transform}}
of $X/G$.

\begin{remark}
In \cite{EdidinRydh}, $X^{\K}/G$ is denoted by $R_G(X,Z)$. 
\end{remark}

Let $\Oscr_{\tilde{X}}(1)$ be the tautological relatively ample line bundle on $\tilde{X}$ and
let $\Oscr_{X^{\K}}(1)$ be its restriction to $X^{\K}$. 
Let $E^R$ denote the exceptional divisor in $X^R$. 
 The following can be extracted from \cite{EdidinRydh}. 
 
\begin{proposition} \label{prop:Reichsteinproperties}
The following properties hold for $X^{\K}$:
\begin{enumerate}
\item\label{1}
$X^{\K}$ has a good quotient $\pi^{\K}:X^{\K}\r X^{\K}\quot G$.
\item\label{2}
The induced map $X^{\K}\quot G\r X\quot G$ is proper. 
\item \label{2.5} $X^{\K}$ satisfies (H1).
\item\label{3} If $X$ satisfies \eqref{H2} then $X^{\K}$ also satisfies \eqref{H2}.
\item\label{4} We have $\mu(X^{\K})<\mu(X)$. 
\item\label{5} For some $N>0$, $\Oscr_{X^{\K}}(N)$ is the pullback of the  line bundle $(\pi^{\K}_\ast \Oscr_{X^{\K}}(N))^G$ on $X^{\K}\quot G$.
\end{enumerate}
\end{proposition}

\begin{proof}
  \eqref{1}, \eqref{2}, \eqref{4} follow from \cite[Theorem 2.11
  (2a),(2c),(3)]{EdidinRydh}.  \eqref{5} follows by Remark
  \ref{lem:torsion}.  The fact that (H1) is true (asserted in \eqref{2.5}) follows from the assumption that
$X$ satisfies (H1) and the fact that $\xi^R:X^R-E^R=(\xi^R)^{-1}(X-Z)\r X-Z$ is an isomorphism.
  For \eqref{3} we observe that $X$ and $X^R$ differ in
  codimension $1$ by the exceptional divisor $E^R$. We have to prove that 
  a generic point of $E^R$ is stable. To this end we use 
  reduction to the linear case made possible by Lemma
  \ref{lem:H12slices} and Lemma \ref{lem:etale} below, see Lemma \ref{lem:local}.  
As we will now switch to the notations introduced in those lemmas, the reader is advised to consult \S\ref{subsec:linear} first.

Note that $G$-stability and $G_e$-stability are equivalent. The exceptional divisor $E^R$ is
  given by $W_0\times\PP(W_1)^{ss}$ (see Lemma \ref{lem:local}).
Let $(w_0,w_1)\in W_0\times W_1$ be a generic point. It is $G_e$-stable by (H1). Since $G_e$ acts trivially on $W_0$ this implies 
in particular that $w_1$ is $G_e$-stable.
 By \cite[Proposition 1.31]{Brion}, 
   $\PP(W_1)^s=\PP(W_1^s)$ 
and hence $[w_1]$ is $G_e$-stable. Thus, $(w_0,[w_1])$ is $G_e$-stable as well.
\end{proof}

In the following commutative diagram we summarize the notations that have been introduced up to now and we also introduce some additional ones which should be self explanatory.
\begin{equation}
\label{eq:kirwan}
\xymatrix{
X^{\K}\ar@/^2em/[rr]^{\pi^R}\ar[d]_{\xi^R}\ar[r]&X^{\K}/G\ar[r]^{\pi^{\K}_s}\ar[d]_{\xi^{\K}_s} &X^{\K}\quot G\ar[d]^{\overline{\xi^{\K}}}\\
X\ar@/_2em/[rr]_{\pi}\ar[r]&X/G\ar[r]_{\pi_s} &X\quot G
}
\end{equation}

\subsection{The Kirwan resolution}\label{subsec:Kirwan}
By repeatedly applying the Reichstein transform, the  maximal stabilizer dimension ultimately becomes $0$ by Proposition \ref{prop:Reichsteinproperties}\eqref{4}. Hence we arrive at a commutative 
 diagram
\[
\xymatrix{
\fks/G\ar[r]^{\pi'_s}\ar[d]_{\Xi} &\fks\quot G\ar[d]^{\bar{\Xi}}\\
X/G\ar[r]_{\pi_s} &X\quot G
}
\]
where $\fks/ G$ is a DM stack and hence $\fks\quot G$ has finite quotient singularities. We call $\fks\quot G$ (or perhaps $\fks/G$) the Kirwan (partial) resolution of $X\quot G$. 

\subsection{Reduction of the Reichstein transform to the linear case}\label{subsec:linear} 
Let $H$ be a reductive group.
For a $H$-representation $W$ we choose a decomposition $W=W_0\oplus W_1$ of $H$-representations, where $W_0=W^{H_e}$.  
\begin{lemma}\label{lem:etale}
With notation as in \S\ref{subsec:prelimlin}, \ref{subsec:Kstep}, 
let $x$ be a point with maximal stabilizer dimension. 
Let $S$ be a smooth affine slice at $x$. 
We have $(G\times^{G_x} S)^R=G\times^{G_x} S^{R,G_x}$, where $S^{R,G_x}$ is the Reichstein transform of $S$ with respect to $G_x$, and there is a $G$-equivariant cartesian diagram 
\begin{equation*}
\xymatrix{
G\times^{G_x}S^{R,G_x}\ar[r]\ar[d] &X^R\ar[d]\\
G\times^{G_x}S\ar[r] &X
}
\end{equation*}
in which the horizontal arrows are strongly \'etale.

We have $X_\mu\cap S=S^{G_{x,e}}$. 
Denote $W=T_{x}S$. Then $W_\mu=W_0$ and  
 there is a $G_x$-equivariant cartesian diagram  
\begin{equation*}
\xymatrix{
S^{R,G_x}\ar[r]\ar[d] &W^{R,{G_x}}\ar[d]\\
S\ar[r] &W
}
\end{equation*}
in which the horizontal maps are strongly \'etale. 
\end{lemma}

\begin{proof}
The equality $(G\times^{G_x} S)^R=G\times^{G_x} S^{R,G_x}$ is an easy verification using the equivalence between the categories of $G$-equivariant schemes over $G/G_x$ and $G_x$-equivariant schemes. 

Both diagrams follow from \cite[Proposition 6.6, Diagram (6.6.1)]{EdidinRydh}  and the Luna slice theorem \S\ref{subsec:prelimlin} (as  strongly \'etale morphism is strong  \cite[Definition 6.4]{EdidinRydh}).
The observation that both strong morphisms 
and \'etale morphisms are preserved under pullback yields that the upper arrows in the diagrams are strongly \'etale.
\end{proof}

In the case of a representation the {Reichstein transform} has a more concrete description recorded in the following lemma. 

\medskip

Let  $\PP(W_1)^{ss}$, $\PP(W_1)^{\u}$ be respectively the semi-stable part of 
$\PP(W_1)$ and its complement, corresponding to the $G$-linearisation $\Oscr(1)$. 
As $\PP(W_1)=(W_1-\{0\})/G_m$ we 
alternatively have  $\PP(W_1)^{ss}=W_1^{ss,1}/G_m$, where $1\in X(G_m)=\ZZ$ is the identity character. If $W^{\text{null}}_1$ is the $G$-nullcone in $W_1$ then
$W_1^{ss,1}=W_1\setminus W_1^{\text{null}}$.

\begin{lemma}\label{lem:local}
If $X=W$ is a representation,   
 $X^{\K}=\underline{\operatorname{Spec}}(\Sym_{W_0\times \PP(W_1)^{ss}}(\Oscr(1)))$.
\end{lemma}

\begin{proof}
We have $\tilde{X}=\underline{\Spec}(\Sym_{W_0\times \PP(W_1)}(\Oscr(1)))$,  $\bar{Z}=W_0\times W_1^{\text{null}}$ and $\tilde{Z}=\underline{\Spec}(\Sym_{W_0\times \PP(W_1)^\ns}(\Oscr(1)))$, which implies the 
claim. 
\end{proof}

We obtain the following diagram
\begin{equation}\label{eq:diagramE}
\xymatrix{E^R/G \ar@/_0.5em/[rr]_{s}\ar[d]^{\pi^R_{E,s}} & &X^R/G\ar[r]^{{\xi^R}}\ar@{.>}@/_0.5em/[ll]_{{\theta}}\ar[d]^{\pi^R_s}&X/ G\ar[d]^{\pi_s}\\
E^R\quot G\ar@/_0.5em/[rr]_{\bar s} & &X^R\quot G\ar[r]^{\overline{\xi^R}}\ar@{.>}@/_0.5em/[ll]_{\bar{\theta}}&X\quot G
}
\end{equation}
where $s$, $\bar{s}$ are obtained from  the inclusion of $E^R$ in $X^R$ and 
where $\theta$ and $\bar\theta$ only exist in the linear case and are obtained from the projection of the line bundle $W^{\K}=\underline{\operatorname{Spec}}(\Sym_{W_0\times \PP(W_1)^{ss}}(\Oscr(1)))$ to the base $W_0\times \PP(W_1)^{ss}$. 

\medskip

From Lemma \ref{lem:local} we obtain a very concrete description of the Reichstein transform in the linear case. 
\begin{lemma}\label{rem:reichsteincovering}
Put $S=\Sym W^\vee$ considered as a graded ring by giving $W_i^\vee$ degree $i$, $i\in \{0,1\}$. 
Then  $W^{\K}\quot G$ is covered by affine charts of the form $\Spec((S^G_f)_{\geq 0})$ for homogeneous $f\in S^G_{>0}$. 
\end{lemma}
\begin{proof}
It follows from Lemma \ref{lem:local} that $W^{\K}$ is covered by affine charts $\Spec((S_f)_{\geq 0})$ for homogeneous $f\in S^G_{>0}$ (by the definition of semi-stable points), and hence $W^{\K}\quot G$ is covered by affine charts as stated.
\end{proof}
\begin{remark}\label{rem:weightedblowup}
Elaborating on Lemma \ref{rem:reichsteincovering} we obtain yet another concrete description of the Reichstein transform in the linear case as a weighted blowup. Let $R$ be a $\ZZ$-graded ring and put $R^\dagger=\oplus_{n\geq 0}R_{\geq n}$, where the right-hand side is $\NN$-graded by the index $n$. 
Then the \emph{weighted blowup} of $\Spec R$ is defined as $\Proj R^\dagger$. 
One easily checks that $W^R\quot G$ is given by the weighted blowup of $\Spec S^G=W\quot G$. (Note that this is not the same as taking the usual blow-up of $W\quot G$ even if $W_0=0$, as $k[W_1]^G$ may not be generated in one degree.)
\end{remark}

\section{The embedding of an NCCR in the Kirwan resolution}\label{sec:embedding}
Let $X$ be as in \S\ref{sec:general}  and assume that $X$ moreover satisfies \eqref{H2}.  Assume we are in the setting of \S\ref{subsec:Kstep}. Let $\Uscr$ be a $G$-equivariant vector bundle on $X$ and define
\[
\Lambda:=\pi_{s\ast} \uEnd_X(\Uscr),\quad
\Lambda':=\pi^{\K}_{s\ast} \uEnd_{X^R}(\xi^{R*}\Uscr).
\] 

\begin{lemma}\label{lem:pushfwd}
If $\Lambda$ is Cohen-Macaulay  (as a sheaf of $\Oscr_{X\quot G}$-algebras) then the canonical morphism 
\begin{equation*}
\label{eq:push}
\Lambda\r R\overline{\xi^R_\ast}\Lambda'
\end{equation*}
is an isomorphism.
\end{lemma}

\begin{proof}
This statement is local for the \'etale topology and hence Lemma \ref{lem:etale} allows us to reduce to the case that $X=W$ is a representation  of $G$. Moreover we may assume by Lemma \ref{rem:vecbundle} that  $\Uscr=U\otimes \Oscr_W$. By Lemma \ref{lem:local} we then have $W^{\K}=\underline{\operatorname{Spec}}(\Sym_{W_0\times \PP(W_1)^{ss}}(\Oscr(1)))$. 
Using the diagram \eqref{eq:kirwan}  we see that 
\[
R\overline{\xi^R_\ast}\Lambda'=\pi_{s\ast} R\xi^R_{s*}(\End(U)\otimes \Oscr_{W^{\K}})\,.
\]
Write $\xi_s:\tilde{W}/G\r W/G$ for the map induced from $\xi:\tilde{W}\r W$ and similarly $j_s$ for the inclusion $W^R/G\r \tilde{W}/G$.
Then we have
\begin{align*}
\pi_{s\ast} R\xi^R_{s*}(\End(U)\otimes \Oscr_{W^{\K}})&=\pi_{s\ast} R(\xi_s j_s)_*(\End(U)\otimes \Oscr_{W^{\K}})\\
&=\pi_{s\ast} (\End(U)\otimes R\xi_{s\ast}Rj_{s\ast} \Oscr_{W^R}).
\end{align*}
We will use the standard distinguished triangle for cohomology with support
\[
\Rscr \Gamma_{\tilde{W}-W^R}(\tilde{W},\Oscr_{\tilde{W}})\r \Oscr_{\tilde{W}} \r Rj_{s\ast} \Oscr_{W^R}\r
\] 
which after applying $\pi_{s\ast} (\End(U)\otimes R\xi_{s\ast}(-))$ yields a distinguished triangle
\[
\pi_{s\ast} (\End(U)\otimes R\xi_{s\ast}
\Rscr \Gamma_{\tilde{W}-W^R}(\tilde{W},\Oscr_{\tilde{W}}))
\r \Lambda
\r R\overline{\xi^R_\ast}\Lambda'\r.
\]
It follows that we need to show
\begin{equation}
\label{eq:affinestatement}
\pi_{s\ast} (\End(U)\otimes R\xi_{s\ast}
\Rscr \Gamma_{\tilde{W}-W^R}(\tilde{W},\Oscr_{\tilde{W}}))
=0.
\end{equation}
We may as well 
compute 
$\Gamma(W\quot G,\pi_{s\ast} (\End(U)\otimes R\xi_{s\ast}
\Rscr \Gamma_{\tilde{W}-W^R}(\tilde{W},\Oscr_{\tilde{W}})))$ since $W\quot G$ is affine. 
 We  have
\begin{multline}
\label{eq:rsgamma}
\Gamma(W\quot G,\pi_{s\ast} (\End(U)\otimes R\xi_{s\ast}
\Rscr \Gamma_{\tilde{W}-W^R}(\tilde{W},\Oscr_{\tilde{W}})))=\\
(\End(U)\otimes 
R\Gamma(\tilde{W},\Rscr\Gamma_{\tilde{W}-W^R}(\tilde{W},\Oscr_{\tilde{W}})))^G.
\end{multline}

Let $E^\ns=W_0\times \PP(W_1)^\ns$, $\tilde{E}=W_0\times\PP(W_1)$.
Since $\tilde{W}-W^R=\theta^{-1}(E^\ns)$ we have
\begin{equation}\label{eq:WE}
R\Gamma(\tilde{W},\Rscr\Gamma_{\tilde{W}-W^R}(\tilde{W},\Oscr_{\tilde{W}}))=\bigoplus_{n\geq 0}R\Gamma(\tilde{E},{\mathcal R}{\Gamma}_{E^{\ns}}(\tilde{E},\Oscr_{\tilde{E}}(n))).
\end{equation}

Put $S=\Sym(W^\vee)=\Gamma(W,\Oscr_W)$.
We put a grading on $S$
 by giving $W^\vee_i$ degree $i$ for $i\in\{0,1\}$. 
Let $\omega$ denote the composition 
\[
\Gr(S)\xrightarrow{\tilde{?}} \Qch(\tilde{E})\xrightarrow{\Gamma_\ast}\Gr(S)
\]
where $\Gamma_\ast(\tilde{E},\Mscr)=\bigoplus_{n\in \ZZ} \Gamma(\tilde{E},\Mscr(n))$ and the first (exact) functor is the usual correspondence
between graded $S$-modules and quasi-coherent sheaves on $\tilde{E}$. It is easy to see that $\tilde{?}$ preserves injectives and hence $R\omega=R\Gamma_*\circ \tilde{?}$. 
We have
\[
{\mathcal R}{\Gamma}_{E^\ns}(\tilde{E},\Oscr_{\tilde{E}}(n))=R\Gamma_{W_0\times W_1^{\text{null}}}(W,\Oscr_W)(n)\,\tilde{}. 
\]
Hence the part of degree $n$ of the right-hand side of \eqref{eq:WE} equals 
\begin{align}\label{eq:con13}
(R\Gamma_\ast((R\Gamma_{W_0\times W_1^{\text{null}}}(W,\Oscr_W)(n))\,\tilde{}\, ))_0&=(R\omega(R\Gamma_{W_0\times W_1^{\text{null}}}(W,\Oscr_W)(n)))_0\\\nonumber
&=(R\omega R\Gamma_{W_0\times W_1^{\text{null}}}(W,\Oscr_W))_n.
\end{align}
 There is a distinguished triangle in $D(\Gr(S))$ 
\begin{equation*}\label{eq:disttriaGrQch}
R\Gamma_{S_{>0}}(M) \to M\to R\omega(M)\to,
\end{equation*}
for every $M\in D(\Gr(S))$, which applied to $M=R\Gamma_{W_0\times W_1^{\text{null}}}(W,\Oscr_W) $ 
gives  the distinguished triangle
\begin{multline*}
R\Gamma_{W_0} R\Gamma_{W_0\times W_1^{\text{null}}}(W,\Oscr_W) 
\to R\Gamma_{W_0\times W_1^{\text{null}}}(W,\Oscr_W) \\\to R\omega R\Gamma_{W_0\times W_1^{\text{null}}}(W,\Oscr_W)\to.
\end{multline*}
Since $W_0\subset W_0\times W^{\text{null}}_1$  the first term equals 
\begin{equation}
\label{eq:W0bound}
R\Gamma_{W_0} R\Gamma_{W_0\times W_1^{\text{null}}}(W,\Oscr_W) =R\Gamma_{W_0}(W,\Oscr_W) 
\end{equation}
which is $0$ in degrees $\geq 0$. 
Thus, the right-hand side of \eqref{eq:WE} equals (using \eqref{eq:con13})  $(R\Gamma_{W_0\times W_1^{\text{null}}}(W,\Oscr_W))_{\geq 0}$.  
So \eqref{eq:affinestatement} translates (via \eqref{eq:rsgamma}) into
\[
(\End(U)\otimes R\Gamma_{W_0\times W_1^{\text{null}}}(W,\Oscr_W))^G_{\geq 0}=0
\]
which by e.g. \cite[Lemma 4.1]{vdbtrace} is equivalent to
\[
(R\Gamma_{W_0\quot G}\Lambda)_{\geq 0}=0.
\]
By local duality (see Corollary \ref{cor:dualizing}) and Cohen-Macaulayness of $\Lambda$ 
this reduces to showing $H_{\leq 0}=0$ for $H:=\Hom_{W\quot G}(\Lambda,\omega_{W\quot G})$.  
Note that $H$ 
is reflexive and localization commutes with $\Hom$, so we may reduce to codimension $1$ and replace $W$ by $W^s$ due to \eqref{H2}. 
As $W^s/G\to W^s\quot G$ is finite, it is also proper.  
Therefore we can apply Grothendieck duality for DM stacks \cite[Corollary 2.10]{nironi}\footnote{This will also be the subject of \cite{YekDM}.} Setting $d_1=\dim W_1$, $d=\dim W$ we obtain 
\begin{equation}
\label{eq:gduality}
\begin{aligned}
H&=\Hom_{W^s\quot G}(\pi_{s*}(\End(U)\otimes \Oscr_W),\omega_{W^s\quot G})\\&=\Hom_{W^s/G}(\End(U)\otimes \Oscr_{W^s},\pi_s^!\omega_{W^s\quot G})\\
&=\Hom_{W^s/G}(\End(U)\otimes \Oscr_{W^s},\omega_{W^s/G})\\
&=\Hom_{W^s/G}(\End(U)\otimes \Oscr_{W^s},\wedge^{\dim G}\mathfrak{g}\otimes \wedge^{d} W\otimes \Oscr_{W^s}(-d_1))\\
&=(\End(U)\otimes \wedge^{\dim G}\mathfrak{g}\otimes \wedge^d W\otimes S)^G(-d_1)
\end{aligned} 
\end{equation}
where the third equality is \cite[Theorem 2.22]{nironi}\footnote{\cite[Theorem 2.22]{nironi} is strictly speaking for proper DM stacks. However, this assumption can be circumvented by first applying compactification (with smooth DM stacks). See the first paragraph of the proof of \cite[Theorem 1.1]{Bergh}.}, 
 the fourth equality  follows from Lemma \ref{lem:omega}, and the fifth equality from the hypothesis \eqref{H2}. 
Hence 
$H_{\leq 0}=(\End(U)\otimes \wedge^{\dim G}\mathfrak{g}\otimes \wedge^d W\otimes S)^G_{\leq -d_1}$.
Since $(\End(U)\otimes \wedge^{\dim G}\mathfrak{g}\otimes \wedge^d W\otimes S)^G$ lives in nonnegative degrees, the conclusion follows.  
\end{proof}

Below we let $\hat{\xi}$ be the morphism of ringed spaces
\[{}
\hat{\xi}:(X^R\quot G, \Lambda')\to(X\quot G,\Lambda)
\]
obtained from $\overline{\xi^R}$.

\begin{corollary}\label{cor:ff} Assume $\Lambda$ is Cohen-Macaulay. 
Then 
$L\hat{\xi}^\ast:D(\Lambda)\to D(\Lambda')$ is a full faithful embedding. 
\end{corollary}

\begin{proof}
 Note
that on the level of $\Oscr_{X^R}$-modules, $R\hat{\xi}_\ast$ is just $R\overline{\xi^R_\ast}$.
It is sufficient
to prove that for any $\Fscr\in D(\Lambda)$ the canonical morphism $\Fscr\r R\hat{\xi}_* L\hat{\xi}^\ast\Fscr$ is an isomorphism.

This is a local statement and hence we may assume that $X$ is affine. We may then replace $\Fscr$ by a K-projective resolution $\Pscr^\bullet$ with
projective terms. Then $ L\hat{\xi}^\ast\Fscr=\hat{\xi}^\ast \Pscr^\bullet$ and moreover by Lemma \ref{lem:pushfwd} $\hat{\xi}^\ast \Pscr^\bullet$ consists
of objects acyclic for $\overline{\xi^R_*}$. Since $\overline{\xi^R_*}$ has finite homological dimension we obtain
$R\overline{\xi^R_*} L\hat{\xi}^\ast\Fscr=\overline{\xi^R_\ast} \hat{\xi}^\ast \Pscr^\bullet=\Pscr^\bullet$ where for the last equality we use again Lemma \ref{lem:pushfwd}.
\end{proof}

As an immediate corollary of Corollary \ref{cor:ff} we get the following embedding of $D(\Lambda)$ to $D(X^R/G)$.
\begin{corollary}\label{cor:bric1}
Assume that $\Lambda$ is Cohen-Macaulay. There is a commutative diagram of derived categories
\begin{equation}\label{diagrambric2}
\xymatrix{
D(\Lambda)\ar@{^{(}->}[rr]^{-\Lotimes_\Lambda \Uscr}\ar@{^{(}->}[d]_{L\hat{\xi}^\ast} &&D(X/ G)\ar[d]^{{L\xi^{R*}}}\\
D(\Lambda')\ar@{^{(}->}[rr]_{-\Lotimes_{\Lambda'} \xi^{R*}\Uscr}& &D(X^R/G)
}
\end{equation}
\end{corollary}

\begin{lemma}\label{lem:cm}
If $\Lambda$ is Cohen-Macaulay then $\Lambda'$ is Cohen-Macaulay.
\end{lemma}

\begin{proof}
Since we may check Cohen-Macaulayness \'etale locally, we can reduce by  Lemma \ref{lem:etale} to the linear case. We use the same notations as in the proof 
of Lemma \ref{lem:pushfwd}. Using Lemmas \ref{rem:reichsteincovering},  Lemma \ref{rem:vecbundle}, locally $\Lambda'$ is of the form $(\Lambda_{f})_{\geq 0}$ for  homogeneous $f\in (S^G)_{>0}$. 

We then need to prove that $(\Lambda_{f})_{\geq 0}$ is Cohen-Macaulay. Let $m=\deg f$. 
We put $A=\Lambda_f$. $A$ is Cohen-Macaulay since it is localization of $\Lambda$ which is Cohen-Macaulay. Note that~$A$ contains a Laurent polynomial ring $A'=A_0[f,f^{-1}]$ as a direct summand. 
As an ascending chain of one sided ideals in $A'$ may be extended to an ascending chain of one sided ideals in $A$ we see that $A'$ is noetherian. A similar argument shows that the $A_i[f,f^{-1}]$,
for $0\le i<m$,
are noetherian $A'$-modules and so they are finitely generated. In particular, $A=\oplus_{0\leq i<m} A_i[f,f^{-1}]$ is finitely generated as an $A'$-module.

Since $A$ is Cohen-Macaulay over $A$, it is Cohen-Macaulay over $A'$. As $A=\oplus_{0\leq i<m} A_i[f,f^{-1}]$ the $A'$-summands $A_i[f,f^{-1}]$ of $A$ are also Cohen-Macaulay $A'$-modules. 
Quotienting by the nonzero divisor $f-1$ we see that $A_i$ is a Cohen-Macaulay $A_0$-module for $0\leq i<m$.  
Thus, $A_i[f]$ is a Cohen-Macaulay $A_0[f]$-module  for $0\leq i<m$. 
Therefore $A_{\geq 0}=\oplus_{0\leq i<m}A_i[f]$ is Cohen-Macaulay $A_0[f]$-module and thus $A_{\geq 0}$ is Cohen-Macaulay  (as it is a finitely  generated $A_0[f]$-module). 
\end{proof}

In the following proposition we show that $\Lambda$ as in Lemma
\ref{lem:cm} can be embedded in the smooth Deligne-Mumford stack
obtained by the Kirwan resolution.

\begin{proposition}\label{prop:embedding}
Let $X$ be a smooth $G$-scheme with a good quotient $\pi:X\r X\quot G$ which satisfies in addition \eqref{H2}.\footnote{\eqref{H2} was imposed on in the beginning of \S\ref{sec:embedding} and it was used explicitly in the proof of Lemma \ref{lem:pushfwd} and implicitly  (via Lemma \ref{lem:pushfwd}) in 
Corollaries \ref{cor:ff}, \ref{cor:bric1}.}  Let $\Xi:\fks/G\r X/G$ be the
Kirwan resolution (see \S\ref{subsec:Kirwan}).

Let $\Uscr$ be a $G$-equivariant vector bundle on $X$ and assume that $\Lambda=\pi_{s\ast}\uEnd_{X}(\Uscr)$ is Cohen-Macaulay on $X\quot G$. Put $\Uscr'=\Xi^*\Uscr$, $\Sigma=\pi'_{s\ast}\uEnd_{\fks}(\Uscr')$. There is a commutative diagram of derived categories
\begin{equation}\label{diagram}
\xymatrix{
D(\Lambda)\ar@{^{(}->}[rr]^{-\Lotimes_\Lambda \Uscr}\ar@{^{(}->}[d]_{L\hat{\Xi}^\ast} &&D(X/ G)\ar[d]^{{L\Xi^{*}}}\\
D(\Sigma)\ar@{^{(}->}[rr]_{-\Lotimes_\Sigma \Uscr'}& &D(\fks/G)
}
\end{equation}
where $\hat{\Xi}$ is the induced morphism of ringed spaces $(\fks\quot G,\Sigma)\r (X\quot G,\Lambda)$.
\end{proposition}

\begin{proof}
The commutativity of the diagram and  full faithfulness of the horizontal arrows are straightforward. It remains to show full faithfulness of the left most vertical arrow. 
By construction, $L\hat{\Xi}^\ast$ is the composition of $L\hat{\xi}^\ast$'s  which 
correspond to a single {Reichstein transform}. By Lemma \ref{lem:cm} we are reduced to showing  full faithfulness for a single {Reichstein transform}. 
In that case the conclusion follows by Corollary \ref{cor:ff}.
\end{proof}

\begin{remark}
Note that the rightmost vertical map in the diagram \eqref{diagram} is in general not fully faithful.
\end{remark}

\subsection{Embedding of $D(\Lambda_Z)$ in $D(X^R/G)$}
In this subsection for use below (c.f. \S\ref{sec:orlovsod}) we extract some consequences of the above results and their proofs. 

From the proof of Lemma \ref{lem:pushfwd} we can extract the following. 
\begin{lemma}\label{rem:afterlempushfwd}
Assume that $\Lambda$ is Cohen-Macaulay. 
Let $\Uscr_Z$, $\Uscr_{E^R}$ denote the restrictions of $\Uscr$, $\xi^{R*}\Uscr$ to $Z$, $E^R$, respectively. Let $\pi_{Z,s}:Z/G\to Z\quot G$ be the quotient map
and denote the restriction/ corestriction of $\overline{\xi^R}$ to a map $E^R\quot G\r Z\quot G$ by $\overline{\xi^R_E}$ so that we have
a commutative diagram
\[
\xymatrix{
E^R\quot G \ar[d]_{\overline{\xi^R_E}}\ar@{^(->}[r]^{\bar{s}} & X^R\quot G\ar[d]^{\overline{\xi^R}}\\
Z\quot G \ar[r] & X\quot G 
}
\] 
Put 
\[
\Lambda_Z:=\pi_{Z,s,*}\uEnd_{Z}(\Uscr_Z),\quad 
\Lambda_Z':=\pi_{E,s,*}^R\uEnd_{E^R}(\Uscr_{E^R}).
\]
Then
\[
R\overline{\xi^R_{E*}}\Lambda_{Z}'=\Lambda_Z,
\] 
and moreover on every connected component $Z_i$ of $Z$
\begin{equation}\label{eq:van}
R\overline{{\xi}^R_{E*}}\pi_{E,s,*}^R(\uEnd_{E^R}(\Uscr_{E^R})(l))\mid_{ Z_i\quot G_i}=0
\end{equation}
for $-c_i<l<0$ with $c_i=\codim(Z_i,X)$ and  where $G_i$ is the stabilizer of $Z_i$ for the action of $G$ on the connected components of $Z$.
\end{lemma}
\begin{proof}
This is proved in a similar (but easier) way as Lemma \ref{lem:pushfwd}. We pass to the linear case for a point $x\in Z_i$.  
In this case $c_i=d_1=\dim W_1$ (with the notation as in \S\ref{subsec:linear}).  
Following the steps of the proof one is reduced to showing that 
$(\End(U)\otimes R\omega R\Gamma_{W_0\times W_1^{\text{null}}}(W,\Oscr_W))^G_{l}=0$ for $-d_1< l \leq 0$.  
Then we use the extra fact (after \eqref{eq:W0bound}) that $R\Gamma_{W_0}(W,\Oscr_W)$ is zero in degrees $>-d_1$ (in the proof of Lemma \ref{lem:pushfwd} it was only needed that it is $0$ in degrees $\geq 0$). The proof then further proceeds as the proof of the lemma, where the bound on $l$ again comes in at the end of the proof.
\end{proof}

Lemma \ref{rem:afterlempushfwd} (together with the proof of Corollary \ref{cor:ff})  makes it possible to construct an embedding of $D(\Lambda_Z)$ in $D(X^R/G)$.  Let 
\[\Lambda_{E^R}=\pi^R_{s*}\uREnd_{X^R}(s_*\Uscr_{E^R}).\] 
Let $\xi^R_E:E^R\to Z$ denote the restriction/corestriction of $\xi^R:X^R\to X$.

\begin{corollary}\label{cor:bric3}
Assume $\Lambda$ is Cohen-Macaulay. Then $R\overline{\xi^R_*}\Lambda_{E^R}\cong\Lambda_Z$. Denote by $\hat{\xi}_E$ the following morphism of DG-ringed spaces\footnote{$\Lambda_{E^R}$ is a sheaf of DG-algebras.} (obtained for $\overline{\xi^R}$)\footnote{We silently identify $\Lambda_Z$ with $i_*\Lambda_Z$ for $i:Z\quot G \hookrightarrow X\quot G$.}
\[
\hat{\xi}_E:(X^R\quot G,\Lambda_{E^R}) \to (X\quot G,\Lambda_Z).
\] 
Then there is a commutative diagram of derived categories
\begin{equation}\label{diagrambric4}
\xymatrix{
D(\Lambda_Z)\ar@{^{(}->}[rr]^{-\Lotimes_{\Lambda_Z} \Uscr_Z}\ar@{^{(}->}[d]_{L\hat{\xi}_E^\ast} &&D(Z/ G)\ar[d]^{Rs_*{L\xi_E^{R*}}}\\
D(\Lambda_{E^R})\ar@{^{(}->}[rr]_{-\Lotimes_{\Lambda_{E^R}} s_*\Uscr_{E^R}}& &D(X^R/G)
}
\end{equation}
\end{corollary}

\begin{proof}
As in Corollary \ref{cor:ff}, it follows from $R\overline{\xi^R_*}\Lambda_{E^R}\cong \Lambda_Z$ that $L\hat{\xi}_E^*$ is fully faithful. Moreover, the horizontal arrows are fully faithful and the diagram commutes. Thus, it suffices to prove that $R\overline{\xi^R_*}\Lambda_{E^R}\cong\Lambda_Z$. 

There is a standard exact sequence on $X^R/G$
\begin{equation}\label{eq:eks0}
0\to \xi^{R*}\Uscr(1)\xrightarrow{t} \xi^{R*}\Uscr\to s_*\Uscr_{E^R}\to 0
\end{equation}
obtained by restriction to $X^R$ of a similar sequence valid for any blowup.

Applying $\uRHom_{X^R}(-,s_*\Uscr_{E^R})$ to  \eqref{eq:eks0} we get the distinguished triangle
\begin{multline}\label{eq:dis}
\uREnd_{X^R}(s_{*}\Uscr_{E^R})\to 
\uRHom_{X^R}(\xi^{R*}\Uscr,s_{*}\Uscr_{E^R})\to\\
\uRHom_{X^R}(\xi^{R*}\Uscr,s_{*}\Uscr_{E^R})(-1)\to.
\end{multline}
By adjointness,
\begin{equation}\label{eq:adj}
\uRHom_{X^R}(\xi^{R*}\Uscr,s_{*}\Uscr_{E^R})=
s_*\uRHom_{E^R}(\Uscr_{E^R},\Uscr_{E^R}).
\end{equation}
Applying $R\overline{\xi^R_*}\pi_{s*}^R$ to \eqref{eq:dis} we obtain by Lemma \ref{rem:afterlempushfwd} 
 that  $R\overline{\xi^R_*}\Lambda_{E^R}\cong \Lambda_Z$ as desired. 
\end{proof}

\section{Homologically homogeneous endomorphism sheaves} 
Endomorphism sheaves of vector bundles appeared in \S\ref{sec:embedding} above.
In this section we discuss the local properties of vector bundles whose endomorphism sheaves have good homological properties. 
This will be used in subsequent sections. 
More precisely the ``fullness'' property will be used in the proof of semi-orthogonal decomposition for the Kirwan resolution given in Theorem \ref{thm:sod}
and the ``saturation'' property will be important for the associated geometric interpretation
obtained in Corollary \ref{cor:external}.

\subsection{Equivariant vector bundles in the case of constant stabilizer dimension}\label{sec:conststab}
Let ~$\XZ$ be a $G$-equivariant smooth
$k$-scheme with a good quotient $\pi:\XZ\r \XZ\quot G$.
We assume that the stabilizers $(G_x)_{x\in \XZ}$ have
dimension independent of~$x$. In particular all orbits in $\XZ$ are
closed (as otherwise the closure of a nonclosed orbit would contain a
closed point with higher dimensional stabilizer) and hence all $G_x$
are reductive.

\medskip

For $x\in \XZ$ we let $H_x\subset G_x$ be the pointwise stabilizer of $T_x \XZ/T_x(Gx)$. This is a normal subgroup of $G_x$. Using the Luna slice
theorem one checks that $H_x$ has finite index in $G_x$ and in particular is reductive.

\begin{definition} Let $\Uscr$ be a $G$-equivariant vector bundle on $Z$.
\begin{enumerate}
\item
$\Uscr$ is \emph{saturated} if for every $x\in \XZ$ we have that the $G_x$-representation $\Uscr_x$ is (up to nonzero multiplicities) induced from $H_x$.
\item Assume that $G$ acts with finite stabilizers. Then we say that $\Uscr$ is \emph{full} if for all $x\in \XZ$, $\Uscr_x$ contains all irreducible $G_x$-representations.
\end{enumerate}
\end{definition}

\begin{lemma}\label{lem:full} Assume that $G$ acts with finite stabilizers and that $\Uscr$ is full. Then $\Uscr$ is a projective generator of $\Qch(\XZ/G)$, locally over $\XZ\quot G$,\footnote{$\pi_{s*}\Hom_Z(\Uscr,\Mscr)=0$ for $\Mscr\in \Qch(\XZ/G)$ implies $\Mscr=0$
(see e.g. \cite{VdB04a}).} and hence in particular 
\[
\Qch(\XZ/G)\cong \Qch(\Lambda)
\]
for $\Lambda=\pi_{s*}\uEnd_\XZ(\Uscr)$.
\end{lemma}
\begin{proof} 
We may check this on strong \'etale neighbourhoods of $x\in \XZ$. Therefore we may assume by Luna slice theorem that $\XZ/G=S/G_x$ where $S$ is a 
smooth connected affine slice at $x$, $G_x$ is finite, and that $\Uscr=U\otimes S$, where $U=\Uscr_x$ by Lemma \ref{lem:bundle}. 
Since $\Uscr$ is full, $U$ contains all irreducible representations of $G_x$, hence $\Uscr$ is projective generator.
\end{proof}
\subsection{Homologically homogeneous sheaves of algebras}\label{sec:hhgencodim1}
\begin{definition}[\cite{BH,stafford2008}\label{def:hh}]
A prime affine $k$-algebra $\Lambda$ is \emph{homologically homogeneous} if it is finitely generated as a module over its centre $R$ and if all simple $\Lambda$-modules have the same (finite) projective dimension.

A coherent sheaf $\Ascr$ of algebras on a $k$-scheme $X$ is homologically homogeneous if $\Ascr(U)$ is homologically homogeneous for every connected affine $U\subset X$.
\end{definition}
We refer to the foundational paper \cite{BH} as a general reference for homologically homogeneous rings.
We also recall from \cite[Lemma 4.2]{VdB04}
\begin{lemma} \label{lem:NCCR}
Assume that $X$ is normal with Gorenstein singularities. A NCCR on $X$ is homologically homogeneous.
\end{lemma}
We now assume that $X$ is a smooth $k$-scheme, $G$ is a reductive group acting with a good quotient $\pi:X\r X\quot G$.
We do not assume that $X$ satisfies $(H1)$. 
\begin{definition}\label{def:genincodim1}
Let $\Uscr$ be a $G$-equivariant vector bundle on $X$. 
$\Uscr$ is {\em generator in codimension one} of $\Qch(X/G)$ if 
$\pi_{s\ast} \uHom_X(\Uscr,\Mscr)=0$ for $\Mscr\in \Qch(X/G)$ implies $\codim \Supp_X \Mscr\ge 2$.
\end{definition}
\begin{theorem}\label{thm:saturated} 
Let $Z\subset X$ is the locus of maximal stabilizer dimension. 
Let $\Uscr$ be a $G$-equivariant vector bundle on $X$ such that $\pi_{s\ast} \uEnd_X(\Uscr)$ is homologically homogeneous on $X\quot G$.
Assume that $\Uscr$ is a generator in codimension one. Then
\begin{enumerate}
\item $\Uscr\mid Z$ is saturated.
\item If $G$ acts with finite stabilizers (and hence $Z=X$) then $\Uscr$ is full.
\end{enumerate}
\end{theorem}
\begin{proof} Using Lemmas \ref{lem:H12slices}, \ref{lem:bundle}  we may reduce to the linear case. 
The result then follows from Lemma \ref{lem:allrep0} below.
\end{proof}
The next proposition says that generation in codimension one is automatic if $X$ is particularly nice.
We say that $X$ is \hypertarget{gnc}{{\em generic}} \cite{SVdB}, if the set $X^{\mathbf{s}}$ of stable points with trivial stabilizers satisfies $\codim(X\setminus X^{\mathbf{s}},X)>1$. 
\begin{proposition}\label{prop:NCCRimpliesU1}
If $X$ is \hyperlink{gnc}{generic} then the following holds true.
\begin{enumerate}
  \item\label{sm}  Every nonzero $G$-equivariant vector bundle $\Uscr$ on $X$ is a generator in codimension $1$. 
  \item\label{rt} $\Lambda=\pi_{s*}\uEnd_X(\Uscr)$ is an NCCR if and only if $\Lambda$ is homologically homogeneous and $X\quot G$ is Gorenstein. 
\end{enumerate}
\end{proposition}

\begin{proof}
\begin{enumerate}
\item
By definition $G$ acts freely on $X^{\bold{s}}$. If $\pi_{s\ast} \uHom_X(\Uscr,\Mscr)=0$ then we may restrict to obtain $\pi_{s\ast} (\uHom_X(\Uscr,\Mscr)\mid X^{\bold{s}})=0$ and by descent we get
$\Mscr\mid X^{\bold{s}}=0$. We now use $\codim(X\setminus X^{\bold{s}},X)\ge 2$. 
\item
For $(\Rightarrow)$ we use Lemma \ref{lem:NCCR}
and the fact that an NCCR (in \cite{VdB04}) is defined for Gorenstein schemes.
For $(\Leftarrow)$ we moreover use that $X$ is generic and therefore $\Lambda\cong \uEnd_{X\quot G}(\pi_{s*}\Uscr)$ and $\pi_{s*}\Uscr$ is reflexive  (see e.g. \cite[Lemma 4.1.3]{SVdB}).  Then $\Lambda$ is an NCCR by definition.\qedhere
\end{enumerate}
\end{proof}

\subsection{The linear case}
We fix some notation that will be in use throughout this section. Let $W=W_0\oplus W_1$, $W_0=W^{\G_e}$  be representations of
$\G$ as in \S\ref{subsec:linear}.
Put $S=k[W]$, graded by giving the
elements of $W^\vee_i$ degree $i$. We define $\GpH\supset G_e$ as the pointwise stabilizer of $W_0$ (this is a normal subgroup of $G$).  

\medskip

For arbitrary representation $U$ of $\G$ we denote $M_{\G,W_?}(U):=(U\otimes k[W_?])^\G$ where $W_?\in \{W,W_0,W_1\}$. We may omit $W_?$ or $G$  in the notation if  they are clear from the context. 
The following lemma may be of independent interest.
\begin{lemma}\label{lem:cm'}
 $M(U)$ is Cohen-Macaulay $S^G$-module if and only if $M_{\GpH}(U)$ is Cohen-Macaulay $S^{\GpH}$-module. 
\end{lemma}
\begin{proof}
  The if direction follows by applying $(-)^{\G/\GpH}$ so we 
  concentrate on the only if direction and assume that $M(U)$ is
a  Cohen-Macaulay $S^\G$-module.

We have as $G/\GpH$-representations
\begin{equation}
\label{eq:covariants}
\begin{aligned}
M_{\GpH}(U)&=k[W_0]\otimes M_{\GpH,W_1}(U)\\
&=\bigoplus_{V,V'\in \rep(\G/\GpH)} (M_{W_0}(V^*)\otimes V)\otimes (M_{W_1}(U\otimes V')\otimes V'^*),
\end{aligned}
\end{equation}
where in the second line we use e.g.
\begin{align*}
M_{\GpH,W_1}(U)&=(U\otimes k[W_1])^{\GpH}\\
&=\bigoplus_{V'\in \rep(G/\GpH)} ((U\otimes k[W_1])^{\GpH}\otimes V')^{\G/\GpH}\otimes V^{\prime\ast}\\
&=\bigoplus_{V'\in \rep(G/\GpH)} ((U\otimes V'\otimes k[W_1])^{\GpH})^{\G/\GpH}\otimes V^{\prime\ast}\\
&=\bigoplus_{V'\in \rep(G/\GpH)} (U\otimes V'\otimes k[W_1])^{\G}\otimes V^{\prime\ast}.
\end{align*}

Applying $\G/\GpH$ to \eqref{eq:covariants} we obtain as $S^G$-modules
\begin{equation}\label{M(U)}
M(U)=\bigoplus_{V\in \rep(\G/\GpH)}M_{W_0}(V^*)\otimes M_{W_1}(U\otimes V).
\end{equation}
As $k[W_0]^{\G/\GpH}\otimes k[W_1]^G=k[W_0]^{\G/\GpH}\otimes (k[W_1]^{\GpH})^{G/\GpH}
\subset (k[W_0]\otimes k[W_1]^{\GpH})^{G/\GpH}
=k[W]^\G$ is a finite
extension of rings and by hypotheses $M(U)$ is a Cohen-Macaulay $k[W]^\G$-module,
$M(U)$ is also a Cohen-Macaulay
$k[W_0]^{\G/\GpH}\otimes k[W_1]^G$-module. By \eqref{M(U)},
$M_{W_0}(V^*)\otimes M_{W_1}(U\otimes V)$ is then a Cohen-Macaulay
$k[W_0]^{\G/\GpH}\otimes k[W_1]^\G$-module.

Consequently, $M_{W_1}(U\otimes V)$ is a Cohen-Macaulay $k[W_1]^G$-module if $M_{W_0}(V^*)\neq 0$ by \cite[Theorem (2.2.5)]{GotoWatanabe}. Since $\G/\GpH$ acts faithfully on $W_0$ (this is the
point where the definition of $\GpH$ is used) it follows (e.g by the proof of \cite[Theorem II.7.1]{Alperin}) that $M_{W_0}(V^*)=M_{\G/\GpH,W_0}(V^*)\neq 0$ for every $V\in \rep(\G/\GpH)$. Thus, 
\[M_{\GpH,W_1}(U)=\bigoplus_{V\in \rep(\G/\GpH)}M_{W_1}(U\otimes V)\otimes V^*\] is a Cohen-Macaulay $k[W_1]^\G$-module. 
Hence $M_{\GpH}(U)=k[W_0]\otimes M_{\GpH,W_1}(U)$ is a Cohen-Macaulay $k[W_0]\otimes k[W_1]^{\G}$-module. Since $k[W_0]\otimes k[W_1]^{\G}\subset k[W]^{\GpH}$ is a finite ring extension, $M_{\GpH}(U)$ is a Cohen-Macaulay $k[W]^{\GpH}$-module.
\end{proof}
The following lemma gives a necessary condition for $\Lambda=\End_{\G,S}(U\otimes S)$ to be homologically homogeneous.
\begin{lemma}\label{lem:allrep0} 
Assume that $\Lambda=\End_{\G,S}(U\otimes S)$ is homologically homogeneous and that $U\otimes \Oscr_W$ is a generator of $\Qch(W/\G)$ in codimension $1$.  
Then $U$ and $\Ind_{\GpH}^\G \Res^{\G}_{\GpH}U$ contain the same irreducible $G$-representations. Moreover, if $\G$ is finite then $U$ contains all irreducible $\G$-representations.
\end{lemma}
\begin{proof}
Let $P=\Ind_{K}^\G\Res_{K}^{G} U\otimes S$ for
$K=\GpH$, or alternatively  $K$ may be the trivial group if $\G$ is finite, 
and put  $Q=U\otimes S$. Consider the evaluation map of $(G,S)$-modules
\[
\phi: \Hom_{\G,S}(Q,P)\otimes_{\Lambda} Q\to P.
\]
We will prove below that $\phi$ is an isomorphism and hence in particular surjective.  Assuming this is the case then by writing $\Hom_{\G,S}(Q,P)$
as a quotient of $\Lambda^{\oplus N}$ as right $\Lambda$-module we find that $P$ is a quotient of $Q^{\oplus N}$ as $(G,S)$-module. Tensoring with $S/S_{>0}$ we obtain
that $\Ind_{K}^\G\Res_{K}^G U$ is a quotient of $U^{\oplus N}$. Since $U$ is a summand of $\Ind_{K}^\G\Res_{K}^{G} U$
this proves that $U$ and $\Ind_{K}^\G\Res_{K}^{G} U$ contain the same irreducible $G$-representation.

\medskip

Now we turn to proving that $\phi$ is an isomorphism. If $G$ is finite then $Q$ is a projective $\Lambda$-module by the  Auslander-Buchsbaum formula \cite[Proposition 2.3]{IyamaReiten}
and if $K=\GpH$ then
we show in the next paragraph that $\Hom_{\G,S}(Q,P)$ is a projective
$\Lambda$-module. Hence in both cases ($G$ finite or $K=\GpH$) $\phi$ is a map between reflexive
$(G,S)$-modules. Now the  kernel and the cokernel of the evaluation map
are supported in codimension $2$ as $Q$ is by assumption a generator in codimension~$1$. 
Hence $\phi$ is an isomorphism.

\medskip

Suppose now $K=\GpH$. Let $V=\Hom(U,\Ind_{\GpH}^G\Res_{\GpH}^{G}U)$. Then 
$\Hom_{\G,S}(Q,P)=M(V)$ and as promised we have to show that $M(V)$ is a projective $\Lambda$-module. 
Now note that $\Res_{\GpH}^G\Ind_{\GpH}^G\Res_{\GpH}^{G}U$ and
$\Res_{\GpH}^GU$ are equal up to multiplicities  (see e.g. the proof of \cite[Lemma 4.5.1]{SVdB}). 
Using the definition of $V$ it then follows easily that $\Res^G_{\GpH}V$ and $\Res^G_{\GpH} \End(U)$ are equal 
up to multiplicities. Since $\Lambda=M(\End(U))$ is assumed to be homologically homogeneous,
it is in particular a Cohen-Macaulay $S^G$-module (e.g. \cite[Theorem 2.3]{stafford2008}) and hence it follows by Lemma \ref{lem:cm'} 
that $M_{\GpH}(\End(U))$ is a Cohen-Macaulay $S^{\GpH}$-module. Thus the same is true
for $M_{\GpH}(V)$. Using Lemma \ref{lem:cm'} again we obtain that
$M(V)$ is Cohen-Macaulay $S^G$-module. Hence $M(V)$ is a projective
$\Lambda$-module, using \cite[Proposition 2.3]{IyamaReiten} again.
\end{proof}

\section{Semi-orthogonal decomposition}\label{sec:orlovsod}
We assume that $X$ is as \S\ref{sec:general} and that $X$ in addition satisfies \eqref{H2}. Assume we are in the setting of \S\ref{sec:setting}, in particular \S\ref{subsec:Kstep}. 

\medskip

In \S\ref{sec:embedding} we embedded $D(\Lambda)$ for $\Lambda=\pi_{s\ast}\uEnd_X(\Uscr)$ in $D(X^{\K}/G)$ via $D(\Lambda')$ for $\Lambda'=\pi^{\K}_{s\ast} \uEnd_{X^R}(\xi^{R*}\Uscr)$. 
If $\Lambda$ is Cohen-Macaulay then so is $\Lambda'$ by Lemma \ref{lem:cm}, and this enabled us to embed $D(\Lambda)$ in $D(\fks/G)$. However,  a similar statement for finite global dimension is
not true. The reader may consult \S\ref{sec:example} for an explicit counterexample. This hampers the inductive construction of semi-orthogonal decomposition of $D(\fks/G)$ with $D(\Lambda)$ as a component. 
In order to remediate the situation we need to tweak the vector 
bundle $\xi^{R*}\Uscr$ by adding suitable twists.

\medskip

Let $\Uscr$ be a vector bundle on $X/G$   
and let $N$ be as in Proposition \ref{prop:Reichsteinproperties}\eqref{5}. 
Put 
\begin{equation}
\label{eq:tweak}
{\Uscr^R}=
\bigoplus_{i=0}^{N-1} (\xi^{R\ast} \Uscr)(i), \quad 
\Lambda^{\K}=\pi^{\K}_{s\ast} \uEnd_{X^{\K}}(\Uscr^{\K}).
\end{equation}
\begin{remark}\label{rem:shiftU}
The most important property of $\Uscr^R$ is that $\Uscr^R(1)\cong \Uscr^R$ locally over $X^R\quot G$. This follows by the definition of $N$. 
\end{remark}

The advantage of $\Uscr^R,\Lambda^R$ (in contrast to $\Uscr',\Lambda'$) is that they inherit more favorable properties from $\Uscr$, $\Lambda$.

\subsection{Some subcategories of $D(X^R/G)$}\

\subsubsection{Local generators}
We slightly generalise some definitions and results from \cite[\S 3.5]{SVdB3}. 
Let $Y$, $Y'$ be  smooth $G$-varieties such that good quotients $Y\to Y\quot G$, $Y'\to Y'\quot G$, respectively, exist. Let $\phi:Y'\to Y$ be  a $G$-equivariant map, denote by   $\bar{\phi}:Y'\quot G\to Y\quot G$ the corresponding quotient map. (In our application below $\bar{\phi}$ will be proper.)
For open $U\subset Y\quot G$ 
we write $\tilde{U}=U\times_{Y\quot G} Y'\subset Y'$. 
 
\begin{definition}
\label{def:locgen} Let $(E_i)_{i\in I}$
be a collection of objects in $D(Y'/G)$. The category $\Dscr$  \emph{locally generated over $Y\quot G$} by $(E_i)_{i\in I}$ is the full subcategory of $D(Y'/G)$ 
spanned by all objects $\Fscr$ such that for every affine open
$U\subset Y\quot G$ 
the object $\Fscr{|}\tilde{U}$ 
is in the  subcategory 
of $D(\tilde{U}/G)$
generated\footnote{Assume $\Tscr$ is a triangulated category closed under coproduct. Let $\Sscr=(T_i)_{i\in I}$ be a set of objects in $\Tscr$. Then the subcategory of $\Tscr$ {\em generated} by $\Sscr$ is the smallest triangulated subcategory of $\Tscr$ closed under isomorphism and coproduct which contains $\Sscr$.} by  $(E_i{|}\tilde{U})_i$. 
 We use the notation $\Dscr=\langle E_i\mid i\in I\rangle^\lll_{Y\quot G}$. 
\end{definition}

The objects $F,H\in D(Y'/G)$ are \emph{locally isomorphic over $Y\quot G$} if there exists a covering $Y\quot G=\bigcup_{i\in I} U_i$ 
such that $F{\mid} \tilde{U}_i\cong H{\mid} \tilde{U}_i$ 
 for all $i$.

 In loc. cit.\ we only considered the case $Y'=Y$, $\phi=\id$ (so 
 there was no ``over $Y\quot G$''). The proofs of the following
 analogues of the results from loc. cit.\ remain valid in this slightly
 more general setting; note only that instead of $\pi_{s*}$ for
 $\pi_s:Y'/G\to Y'\quot G$ we use $R\bar{\phi}_*\pi_{s*}$ (taking into account that $\phi$ is now not the identity)
and that in loc.\ cit.\ we used small categories, instead of the large, cocomplete, categories we are using here.
\begin{lemma}\cite[Lemma 3.5.3]{SVdB3}\label{lem:onecovering}
 Let $(E_i)_{i\in I}$ be a collection of perfect objects
  in $D(Y'/G)$ and let 
   {$\Fscr\in D(Y'/G)$}.  Let $Y\quot
  G=\bigcup_{j=1}^n U_j$ 
   be a finite open affine covering of $Y\quot G$.  
    If
  for all $j$ one has that $\Fscr{\mid}\tilde{U}_j$ is in the 
subcategory of $D(\tilde{U}_j/G)$ 
  {generated by}
  $(E_i{|}\tilde{U}_j)_i$ then $\Fscr$ is in the subcategory of $D(Y'/G)$
locally 
generated over $Y\quot G$ 
by $(E_i)_{i\in I}$.
\end{lemma}

The following result shows that semi-orthogonal decompositions can be constructed locally.
\begin{proposition}\cite[Proposition 3.5.8]{SVdB3}
\label{th:recognition}
Let $I$ be a  totally ordered set. Assume 
{$\Dscr\subset D(Y'/G)$} is locally  generated over $Y\quot G$ by
a collection of subcategories $\Dscr_i$
closed under 
{coproduct and} local isomorphism over $Y\quot G$.
Assume that $R\bar{\phi}_*\pi_{s\ast}\uRHom_{Y'}(\Dscr_i,\Dscr_j)=0$ for $i>j$. Then $\Dscr$ is generated by $(\Dscr_i)_i$
and in particular we have a semi-orthogonal decomposition $\Dscr=\langle \Dscr_i\mid i\in I\rangle$.
\end{proposition}

It is convenient to pick for every $E\in D(Y'/G)$ a $K$-injective resolution (with injective terms)  
$E\r I_E$ and to represent $R\bar{\phi}_*\pi_{s\ast}\uRHom_{Y'}(E,F)$ on $Y\quot G$ by the complex of sheaves
$U\mapsto \bar{\phi}_*\Hom_{\tilde{U}}(I_E{\mid}\tilde{U},I_F{\mid}\tilde{U})^G$.\footnote{It is enough to show that each term in $\uHom_{Y'}(I_E,I_F)$, and consequently in $\uHom_{Y'}(I_E,I_F)^G$, is flabby, 
since flabby sheaves are acyclic for $\bar{\phi}_*$ and $\bar{\phi}_*$ has finite cohomological dimension.  
Let $\Fscr,\Iscr\in \Qch(Y'/G)$ with $\Iscr$ injective.  
Let $j:U\hookrightarrow Y'$ be an open immersion.   
As $j_!j^*\Fscr\hookrightarrow \Fscr$ and $\Iscr$ is injective in $\Mod(Y'/G)$ (as $Y'$ is noetherian),  
$\uHom_{Y'}(\Fscr,\Iscr)\twoheadrightarrow\uHom_X(j_!j^*\Fscr,\Iscr)=\uHom_{U'}(j^*\Fscr,j^*\Iscr)=\uHom_{U'}(\Fscr\mid_{U'},\Iscr\mid_{U'})$.
} 
With this representation
\[\Lambda:=R\bar{\phi}_*\pi_{s\ast}\uREnd_{Y'}(E):=\allowbreak R\bar{\phi}_*\pi_{s\ast} \uRHom_{Y'}(E,E)\]
is a sheaf of DG-algebras
on $Y\quot G$
and $R\bar{\phi}_*\pi_{s\ast}\uRHom_{Y'}(E,F)$ is a sheaf of right $\Lambda$-DG-modules.

\begin{lemma}\cite[Lemma 3.5.6]{SVdB3} \label{lem:ff}
Assume that $\Dscr\subset D(Y'/G)$ is locally  
generated over $Y\quot G$ by the perfect complex $E$. Let $\Lambda=R\bar{\phi}_*\pi_{s\ast}\uREnd_{Y'}(E)$ be the sheaf of DG-algebras on $Y\quot G$ as defined above.
The functors
\begin{align*}
\Dscr\r D(\Lambda)&:F\mapsto R\bar{\phi}_*\pi_{s\ast}\uRHom_{Y'}(E,F),\\
D(\Lambda)\r \Dscr&:H\mapsto \bar{\phi}^{-1}H\Lotimes_{\bar{\phi}^{-1}\Lambda} E
\end{align*}
are well-defined (the second functor is computed starting from a $K$-flat resolution\footnote{Such a $K$-flat resolution is constructed in the same way as for DG-algebras
(see \cite[Theorem 3.1.b]{Keller1}). One starts from the observation that for every $M\in D(\Lambda)$ there is a morphism $\bigoplus_{i\in I} j_{i!}(\Lambda{\mid} U_i)\r M$
with open immersions $(j_i:U_i\r X\quot G)_{i\in I}$, which is an epimorphism on the level of cohomology.
 }
 of $H$) and yield inverse equivalences between $\Dscr$ and $D(\Lambda)$.
\end{lemma}

\subsubsection{Locally generated subcategories}\label{subsubsec:localgensubcat}
We define some locally generated subcategories of $D(X^R/G)$ which we  will need for the semi-orthogonal decomposition of $D(\Lambda^R)$. 

Let 
\[
\Cscr_{X^R}:=\langle \Uscr^R\rangle^\lll_{X^R\quot G}\subset D(X^R/G).
\]
Our aim will be to define a semi-orthogonal decomposition of $\Cscr_{X^R}$.

Let
\begin{equation}\label{eq:bric6}
\tilde{\Cscr}_X:=\langle \xi^{R*}\Uscr\rangle^\lll_{X\quot G}\subset \Cscr'_{X^R}:=\langle \xi^{R*}\Uscr\rangle^\lll_{X^R\quot G}\subset \Cscr_{X^R}.
\end{equation}
Let $Y$ be a connected component of $Z$, $G_Y$ be the stabilizer of $Y$ in $Z$ (not pointwise), let $E^R_Y=(\xi^{R})^{-1}(Y)$ and let $s_Y:E^R_Y/G_Y\to X^R/G$ be the inclusion. Let $\Uscr_Y$, $\Uscr_{E^R_Y}$ be the restrictions of $\Uscr$, $\xi^{R*}\Uscr$ to $Y$, $E^R_Y$, respectively.
 Let $\pi_{Y,s}:Y/G_Y\to Y\quot G_Y$,
$\pi_{E^R_Y,s}:E^R_Y/G_Y\to E^R_Y\quot G_Y$,
 be the quotient maps.  
Let 
\[
\Cscr_{Y,n}=\langle s_{Y*}\Uscr_{E^R_Y}(n)\rangle^\lll_{X\quot G}.
\]
\begin{lemma}\label{lem:cyn}
We have $\Cscr_{Y,n}\subset \Cscr_{X^R}$ for all $n\in \ZZ$. 
\end{lemma}
\begin{proof}
Recall the standard exact sequence \eqref{eq:eks0} on $X^R/G$
\begin{equation}\label{eq:eks}
0\to \xi^{R*}\Uscr(1)\xrightarrow{t} \xi^{R*}\Uscr\to s_*\Uscr_{E^R}\to 0.
\end{equation}

Let $n\in \ZZ$. 
Since $\xi^{R*}\Uscr(n)$ belongs to $\Cscr_{X^R}$  by Remark \ref{rem:shiftU} (and the definition of $\Uscr^R$), it follows from (twisting) \eqref{eq:eks} that $s_*\Uscr_{E^R}(n)\in \Cscr_{X^R}$. 
As the (local) generator  $s_{Y*}\Uscr_{E^R_Y}(n)$ of $\Cscr_{Y,n}$ is its direct summand (identifying $Y/G_Y=(\cup_{g\in G}gY)/G$, $E^R_Y/G_Y=E^R_{\cup_{g\in G}gY}/G$) the lemma follows. 
\end{proof}

\begin{proposition}\label{prop:sodC}
Assume that $\Lambda$ is Cohen-Macaulay. 
Let $Z_1,\dots,Z_t$ be representatives for the orbits of the $G$-action on the connected components of $Z$. There is a semi-orthogonal decomposition
\[\Cscr_{X^R}=\langle \tilde{\Cscr}_X,\Cscr_{Z_1,0},\dots, \Cscr_{Z_1,c_1-2},\dots,\Cscr_{Z_t,0},\dots, \Cscr_{Z_t,c_t-2}\rangle,\]
where $c_i=\codim(Z_i,X)$. Moreover, the components corresponding to different $Z_i$ are orthogonal. 
\end{proposition}

\begin{proof}
By \eqref{eq:bric6}, Lemma \ref{lem:cyn}, we have $\tilde{\Cscr}_X,\Cscr_{Z_i,n}\subset \Cscr_{X^R}$, respectively.  
We apply Proposition \ref{th:recognition}.  

By definition of
$\Cscr_{Z_i,n}$ it is clear that the components corresponding to
  different $i$ are orthogonal.
    To obtain the orthogonality for $\Cscr_{Z_i,n}$ for fixed $i$, 
    let us first recall \eqref{eq:adj} for an easier reference
    \begin{equation}\label{eq:adj1}
\uRHom_{X^R}(\xi^{R*}\Uscr,s_{*}\Uscr_{E^R})=
s_*\uRHom_{E^R}(\Uscr_{E^R},\Uscr_{E^R}).
\end{equation} 
Applying
  $\uRHom_{X^R}(-,s_*\Uscr_{E^R}(n))$ to \eqref{eq:eks}\footnote{We obtain  an analogue of \eqref{eq:dis}.} and
  using \eqref{eq:adj1} it is then enough to show that
\[
R\overline{\xi^R_*}\pi_{s*}(s_{Z_i*}\uRHom_{E^R_{Z_i}}(\Uscr_{E^R_{Z_i}},\Uscr_{E^R_{Z_i}}(l)))=0
\]
for $-(c_i-2)-1\le l<0$. This holds by Lemma \ref{rem:afterlempushfwd}. 

To obtain the semi-orthogonality of $\Cscr_{Z_i,n}$ for $n=0,\ldots,c_i-2$
and $\tilde{\Cscr}_X$ we need 
\begin{equation}\label{eq:nog1}
R\overline{\xi^R_*}\pi_{s*}\uRHom_{X^R}(s_{Z_i*}\Uscr_{E^R_{Z_i}}(n),\xi^{R*}\Uscr)=0
\end{equation}
for $l$ in the indicated range.

We first apply $\uRHom_{X^R/G}(-,\xi^{R*}\Uscr)$ to \eqref{eq:eks} and obtain the distinguished triangle
\begin{multline*}
\uRHom_{X^R}(s_*\Uscr_R,\xi^{R*}\Uscr)\to \uRHom_{X^R}(\xi^{R*}\Uscr,\xi^{R*}\Uscr)\\\xrightarrow[]{\uRHom_{X^R}(\xi^{R*}\Uscr\otimes t,\xi^{R*}\Uscr)}
\uRHom_{X^R}(\xi^{R*}\Uscr(1),\xi^{R*}\Uscr)\to
\end{multline*}
where $t:\Oscr_{X^R/G}(1)\to \Oscr_{X^R/G}$ denotes the canonical map. 
By Lemma \ref{lem:square} below  
this may be rewritten as 
\begin{multline*}
\uRHom_{X^R}(s_*\Uscr_R,\xi^{R*}\Uscr)\to \uRHom_{X^R}(\xi^{R*}\Uscr,\xi^{R*}\Uscr)\\\xrightarrow[]{\uRHom_{X^R}(\xi^{R*}\Uscr,\xi^{R*}\Uscr\otimes t)}
\uRHom_{X^R}(\xi^{R*}\Uscr,\xi^{R*}\Uscr(-1))\to.
\end{multline*}
By applying $\uRHom_{X^R/G}(\xi^{R*}\Uscr,-(-1))$ to \eqref{eq:eks}
 we then deduce that 
 \[
 \uRHom_{X^R}(s_*\Uscr_{E^R},\xi^{R*}\Uscr)\cong \uRHom_{X^R}(\xi^{R*}\Uscr,s_*\Uscr_{E^R}(-1))[-1].
 \]
 Twisting and applying \eqref{eq:adj1} we moreover have
  \[
 \uRHom_{X^R}(s_*\Uscr_{E^R}(n),\xi^{R*}\Uscr)\cong s_*\uRHom_{E^R}(\Uscr_{E^R},\Uscr_{E^R}(-1-n))[-1].
 \]
Applying $R\overline{\xi^R}_*\pi_{s*}$ and using Lemma \ref{rem:afterlempushfwd}, we obtain \eqref{eq:nog1}.

We now prove the generation property. 
We reduce to the affine $X$ containing one representative $Z_j$ of connected components of $Z$ by Proposition \ref{th:recognition} (and Lemma \ref{lem:onecovering}). By \eqref{eq:eks}, it follows that $\langle \tilde{\Cscr}_X,\Cscr_{Z_j,0},\dots, \Cscr_{Z_j,c_j-2}\rangle$ contains $\xi^{R*}\Uscr(i)$ for $0\leq i\leq c_j-1$. 
By the proof of \cite[Lemma 3.2.2]{VdB04a} (which applies in the $G$-equivariant setting) it follows that $\Cscr_{X^R}$ contains $\xi^{R*}\Uscr(i)$ for all $i\in \ZZ$. We now show that $(\xi^{R*}\Uscr(i))_{i\in \ZZ}$ generate $\Cscr_{X^R}$. 
Let $0\neq \Fscr\in D(X^R/G)$. 
We need to show that 
\begin{multline*}
(\RHom_{X^R/G}(\xi^{R*}\Uscr(i),\Fscr)=0 \text{ for all $i\in\ZZ$})\implies \\
\pi^R_{s*}\uRHom_{X^R}(\Uscr^R,\Fscr)=\pi^R_{s*}((\Uscr^R)^\vee\otimes \Fscr)=0
\end{multline*}
or equivalently
\begin{multline*}
(\pi^R_{s*}\uRHom_{X^R}(\Uscr^R,\Fscr)=\pi^R_{s*}((\Uscr^R)^\vee\otimes \Fscr)\neq 0)\implies \\
\RHom_{X^R/G}(\xi^{R*}\Uscr(i),\Fscr)\neq 0 \text{ for some $i\in\ZZ$}
.
\end{multline*}
We assume that $\pi^R_{s*}((\Uscr^R)^\vee\otimes \Fscr)\neq 0$. 
Recall that $\Oscr(N)=\pi^{R*}_s\Mscr$ for an ample line bundle $\Mscr$ on $X^R\quot G$ by Proposition \ref{prop:Reichsteinproperties}\eqref{5}.  
Then $\Hom_{X^R\quot G}(\Mscr(m),\pi^R_{s*}((\Uscr^R)^\vee\otimes \Fscr)) \neq 0$ for $m\ll 0$ (since $X^R\quot G$ is proper over affine $X\quot G$ by Proposition \ref{prop:Reichsteinproperties}\eqref{2}). Thus
\begin{multline*}
0\neq \RHom_{X^R\quot G}(\Mscr(m),\pi^R_{s*}((\Uscr^R)^\vee\otimes \Fscr))= \RHom_{X^R/ G}(\Oscr(mN),(\Uscr^R)^\vee\otimes \Fscr)=\\
\oplus_{0\leq i<N}\RHom_{X^R/G}(\xi^{R*}\Uscr(i+mN),\Fscr)
\end{multline*}
and the generation follows. 
\end{proof}

We have used the following lemma.
\begin{lemma}
\label{lem:square}
Let $\Fscr$, $\Gscr\in \Qch(X^R/G)$ and $t:\Oscr_{X^R/G}(1)\to \Oscr_{X^R/G}$ the canonical map. Then the following diagram is commutative
\[
\xymatrix{
\uRHom_{X^R}(\Fscr,\Gscr)\ar[d]_{\uRHom_{X^R}(\Fscr,t)}\ar@{=}[r]&\uRHom_{X^R}(\Fscr,\Gscr)\ar[d]^{\uRHom_{X^R}(t,\Gscr)}\\
\uRHom_{X^R}(\Fscr,\Gscr(-1))\ar[r]_{\cong}&\uRHom_{X^R}(\Fscr(1),\Gscr)
}
\] 
\end{lemma}

\subsection{Orlov's semi-orthogonal decomposition for the Reichstein transform}
We are now ready to formulate our next main result which is \emph{an analogue for the Reichstein transform of Orlov's semi-orthogonal decomposition for a blowup \cite{Orlov}}.
\begin{theorem}\label{thm:sodC} 
Let $X$ be a smooth $G$-scheme such that a good quotient $\pi:X\r X\quot G$ exists. Assume furthermore that $(X,G)$ satisfies \eqref{H2}.\footnote{\eqref{H2} was imposed at the beginning of \S\ref{sec:orlovsod} and has been used throughout \S\ref{subsubsec:localgensubcat} implicitly via results in \S\ref{sec:embedding}.}
Let $Z\subset X$ be the locus of maximal stabilizer dimension and
let $Z_1,\dots,Z_t$ be representatives for the orbits of the $G$-action on the connected components of $Z$. Let $G_i$ be the stabilizer of $Z_i$. 

\medskip

 Let $\Uscr$ be a $G$-equivariant vector bundle
on $X$  such that $\pi_{s*}\uEnd_X(\Uscr)$ is Cohen-Macaulay,
and put
\[
\begin{array}{rcll}
\Cscr_{X^R}&\coloneqq&\langle \Uscr^R\rangle^\lll_{X^R\quot  G}&\subset D(X^R/G),\\
\Cscr_X&\coloneqq&\langle \Uscr\rangle^\lll_{X\quot  G}&\subset D(X/G),\\
\Cscr_{Z_i}&\coloneqq&\langle \Uscr_{Z_i}\rangle^\lll_{Z_i\quot G_i}&\subset D(Z_i/G_i).
\end{array}
\]
Let $\xi^R_E:E^R\to Z$ denote the restriction/corestriction of $\xi^R:X^R\to X$. 

\medskip

The following holds. 
\begin{enumerate}
\item\label{eq:nog2}  
$L\xi^{R*}:D(X/G)\r D(X^R/G)$ is fully faithful when restricted to $\Cscr_X$.
\item\label{eq:nog3} 
The composition 
\[
F_i:D(Z_i/G_i)\hookrightarrow D(Z/G)\xrightarrow{L\xi_E^{R\ast}} D(E^R/G)\xrightarrow{Rs_\ast} D(X^R/G)
\]
is fully faithful when restricted to $\Cscr_{Z_i}$. 
\item
There is a semi-orthogonal decomposition of $\Cscr_{X^R}$
\[\langle L\xi^{R\ast} \Cscr_X,(F_1\Cscr_{Z_1})(0),\dots, (F_1\Cscr_{Z_1})(c_1-2),\dots,(F_t\Cscr_{Z_t})(0),\dots, (F_t\Cscr_{Z_t})(c_t-2)\rangle
\]
where $c_i=\codim(Z_i,X)$. 
Moreover, the components corresponding to different $Z_i$ are orthogonal. 
\end{enumerate}
\end{theorem}

\begin{proof}
{ }
(1) This follows from Corollary \ref{cor:bric1}. 

(2) This follows from Corollary \ref{cor:bric3}.

(3) In the notation of \S\ref{subsubsec:localgensubcat},   $F_i\Cscr_{Z_i}(n)=\Cscr_{Z_i,n}$, $L\xi^{R*}\Cscr_X=\tilde{\Cscr}_X$.  
The claim then follows immediately from Proposition \ref{prop:sodC}.
 \end{proof}

\begin{corollary}\label{cor:sod} 
Let the notations and assumptions be as in the previous theorem and define in addition sheaves of algebras on $X^R\quot G$, $X\quot G$, $Z_i\quot G_i$ via:
\[
\Lambda^R\coloneqq \pi^R_{s\ast} \uEnd_{X^R}(\Uscr^R), \quad \Lambda\coloneqq \pi_{s\ast} \uEnd_X(\Uscr),\quad \Lambda_{Z_i}\coloneqq \pi_{Z_i,s,\ast} \uEnd_X^R(\Uscr_{Z_i})
\]
where $\pi_{Z_i}:Z_i\r Z_i\quot G_i$ is the good quotient.

\medskip

There is a semi-orthogonal decomposition
\[
D(\Lambda^{\K})\cong \langle D(\Lambda),\underbrace{D(\Lambda_{Z_1}),\ldots,  D(\Lambda_{Z_1})}_{c_1-1},\ldots,
\underbrace{ D(\Lambda_{Z_t}),\ldots,  D(\Lambda_{Z_t})}_{c_t-1}\rangle
\]
where   $c_i=\codim(Z_i,X)$.  
Moreover, the components corresponding to different $Z_i$ are orthogonal. 
\end{corollary}

\begin{proof}
This is an immediate consequence of Theorem \ref{thm:sodC} using Lemma \ref{lem:ff}.

Note that we could also deduce this corollary from Proposition \ref{prop:sodC} and Lemma \ref{lem:ff} together with results from \S\ref{sec:embedding}; i.e. for $\Cscr_{X^R}$, $\Cscr_{Z_i,n}$ we could apply the lemma with $Y'=Y=X^R,\phi=\id$, and then use Corollary \ref{cor:bric3} to further describe the latter, and for $\tilde{\Cscr}_X$ with $Y'=X^R,Y=X,\phi=\xi^R$ followed by Lemma \ref{lem:pushfwd}.  
\end{proof}

\subsection{Properties of  $\Uscr^R,\Lambda^R$ inherited from $\Uscr,\Lambda$}\label{sec:inherit}
For use below we recall that  $\Uscr_{E^R}$ was defined as the restrictions of $\xi^{R*}\Uscr$  to the exceptional divisor $E^R/G$ for the morphism $\xi^R:X^R/G\to X/G$. We similarly let~$\Uscr^R_{E^R}$ be the restriction of $\Uscr^R$
to $E^R/G$.

\begin{lemma}\label{lem:descofLambdaR}
The sheaf of rings on $E^R\quot G$
\[
\bar{\Lambda}:=\oplus_{n=-\infty}^\infty \pi^R_{E,s,*}\uHom_{E^R}(\Uscr^R_{E^R},\Uscr^R_{E^R}(n))
\]
is strongly graded. 
If  in the linear case as in Lemmas \ref{lem:local}, \ref{rem:vecbundle} 
then on $E^R\quot G$
\[
\bar{\theta}_*\Lambda^R\cong \bar{\Lambda}_{\geq 0}
\]
where $\bar{\theta}:X^R\quot G\to E^R\quot G$ \eqref{eq:diagramE} is the splitting of the  inclusion $E^R\quot G\to X^R\quot G$.
\end{lemma}

\begin{proof}
We start by proving that $\bar{\Lambda}$ is strongly graded.\footnote{A $\ZZ$-graded ring $\Gamma$ is strongly graded if $1\in \Gamma_1 \Gamma_{-1}$, $1\in \Gamma_{-1}\Gamma_1$.} 
Since $\Uscr_{E^R}^R(1)\cong \Uscr_{E^R}^R$
 locally over $E^R\quot G$ (restricting the local isomorphism $\Uscr^R(1)\cong \Uscr^R$ over $X^R\quot G$ in Remark \ref{rem:shiftU} to $E^R/G$), $\bar{\Lambda}$ has a unit in degree $1$ and it is thus strongly graded.

We now prove the second statement. We have $\Uscr^R=\theta^*\Uscr^R_{E^R}$ (using the linearity assumption $\Uscr=U\otimes \Oscr_W$). 
We compute
\begin{align*}
\bar{\theta}_*\Lambda^R&=\bar{\theta}_*\pi^R_{s*}\uEnd_{X^R}(\theta^*\Uscr^R_{E^R})\\
&\cong \pi^R_{E,s,*}\uHom_{E^R}(\Uscr^R_{E^R},\theta_*\theta^*\Uscr^R_{E^R})\\
&\cong \oplus_{n=0}^\infty \pi^R_{E,s,*}\uHom_{E^R}(\Uscr^R_{E^R},\Uscr^R_{E^R}(n)).\qedhere
\end{align*}
\end{proof}

\begin{lemma}\label{lem:homologicallyhom} Let $X$ be a scheme and let $A$ be a strongly graded sheaf of algebras on $X$. 
If $A$ is homologically homogeneous then $A_0$
and $A_{\ge 0}$ are  homologically homogeneous  on $X$.
\end{lemma}
\begin{proof} 
This is a local statement so we may assume that $A$ is a strongly graded ring. It is clear that $A[t]$ for $|t|=-1$ is also strongly graded.
Since $A_{\geq 0}\cong A[t]_0$ we need two facts:
\begin{enumerate}{}
\item If $A$ is homologically homogeneous then so is $A[t]$.
\item If $A$ is strongly graded and homologically homogeneous then so is $A_0$.
\end{enumerate}
The first fact is \cite[Theorem 7.3]{BH}. The second follows since the categories of $A_0$-modules and graded $A$-modules are in this case equivalent \cite[Theorem I.3.4]{gradedringtheory}, and by \cite[Proposition 2.9]{stafford2008}. 
\end{proof}

The next proposition exhibits some properties of the pair $(X,\Uscr)$
which lift to the pair $(X^R,\Uscr^R)$. 
\begin{proposition}\label{prop:want}
\begin{enumerate}
\item\label{T1}
If $\Lambda=\pi_{s\ast} \uEnd_X(\Uscr)$ 
is  homologically homogeneous on $X\quot G$
then the same is true for $\Lambda^{\K}=\pi^{\K}_{s\ast} \uEnd_{X^{\K}}(\Uscr^{\K})$. 
\item\label{T2} 
 If $\Uscr$ is generator in codimension one then the same is true for $\Uscr^{R}$.
\end{enumerate}
\end{proposition}

\begin{proof}
\eqref{1}
We  reduce to the linear case by Lemmas \ref{lem:etale}, \ref{rem:vecbundle}. 
Let $\check\Lambda$ be   the sheaf of rings on $X^R\quot G$ defined by 
\[
\bar{\theta}_*\check\Lambda= \bar{\Lambda}=\oplus_{n=-\infty}^\infty \pi^R_{E,s,*}\uHom_{E^R/G}(\Uscr^R_{E^R},\Uscr^R_{E^R}(n))
\]
where the right-hand side is to be viewed as a sheaf of  $\bar{\theta}_\ast\Oscr_{X^R\quot G}$ algebras.
Let 
\[
\Gamma=\oplus_{n=-\infty}^\infty \pi^R_{E,s,*}\uHom_{E^R/G}(\Uscr_{E^R},\Uscr_{E^R}(n)).
\]
Using the definition of $N$ we get
\[
\bar{\theta}_{*}\check\Lambda\cong
\begin{pmatrix}
\Gamma&\Gamma(1)&\cdots&\Gamma(N-1)\\
\Gamma(-1)&\ddots&&\vdots\\
\vdots&&&\\
\Gamma(-N+1)&\cdots&&\Gamma
\end{pmatrix}
\]
as sheaves of $\ZZ$-graded algebras on $E^R\quot G$, where $?(i)$ denotes the grading shift. 

Let $R=k[W]^G$ with $W$ graded as in Lemma \ref{rem:reichsteincovering}. 
Note that $\Proj R\cong E^R\quot G$. 
Let $\hat\Lambda$ be the sheaf of graded $\Oscr_{E^R\quot G}$-algebras
 associated to the $k[W]^G$-algebra $\Lambda$, defined by $\hat\Lambda(U_f)=\Lambda_f$, where $U_f=\Spec((R_f)_0)$ for $f\in R_{>0}$.  

 We claim that $\hat{\Lambda}\cong{\Gamma}$. 
 We will now confuse quasi-coherent sheaves on affine schemes with their global sections. 
 First note that $\Lambda=(\End(U)\otimes k[W])^G$. 
 Hence $\hat{\Lambda}(U_f)=\Lambda_f=(\End(U)\otimes k[W]_f)^G$. 
 On the other hand,
 \begin{align*}
 \Gamma(U_f)&=
 \oplus_{n=-\infty}^\infty \pi_{E,s,*}^R\uHom_{E^R/G}(U\otimes \Oscr_{E^R},U\otimes \Oscr_{E^R}(n))(U_f)\\
 &=(\End(U)\otimes\oplus_{n=-\infty}^\infty\Gamma(E^R_f,\Oscr(E^R)(n)))^G\\
 &=(\End(U)\otimes \oplus_{n=-\infty}^\infty(k[W]_f)_n)^G\\
 &=(\End(U)\otimes k[W]_f)^G.
 \end{align*}

Since $\Lambda$ is homologically homogeneous, so is $\hat{\Lambda}$ and therefore $\Gamma$. Thus, $\bar{\theta}_*\check\Lambda$ is homologically homogeneous. Since $\bar{\theta}_*\check\Lambda$ is strongly graded by Lemma \ref{lem:descofLambdaR} and $\bar{\theta}_*\Lambda^R=(\bar{\theta}_*\check\Lambda)_{\geq 0}$ by Lemma \ref{lem:descofLambdaR}, $\Lambda^R$ is homologically homogeneous by Lemma \ref{lem:homologicallyhom}.

\eqref{T2} 
 Also \eqref{T2} can be checked \'etale locally, so we can reduce to the linear case by Lemmas \ref{lem:etale}, \ref{rem:vecbundle}. 
 Since $X$ and $X^R$ differ in codimension $1$ by the exceptional divisor $E^R$, we need to show that, generically  on  $E^R$, 
$\Uscr^R$ generates $D(E^R/G)$. 
It is enough to check that $\Uscr^R_y=\Uscr^R\otimes k(y)$ contains all the irreducible representations of $G_y$ for a generic point $y\in E^R$. 
Denote $H=G_y$ and let $x$ a generic point in $W$ such that $y=[x]$. 
 Then $H$ (not necessarily pointwise) stabilizes the line $\ell$ passing through $0$, $x$, and the action of $H$ on $\ell$ is then given by a character $\alpha\in X(H)$.  
 Let $K=\ker
 \alpha$, which is the (finite) stabilizer of $x$. Thus,
 $H/K$
 can be considered as a subgroup of $G_m$,
 and it is therefore cyclic or $G_m$.
 However, if the $H/K$
 were $G_m$
 then $0$
 would be in the closure of the orbit of $x$.
 This is a contradiction since $x$
 is generic in $W$
 and $W$
 satisfies \eqref{H2}.
 Thus, $H/K$
 is a finite cyclic group and it acts on the line $\ell$
 by a generator of $X(H/K)$.
 Since $\Uscr^R_y=\oplus_{i=0}^{N-1}
 U\otimes (\ell^*)^{\otimes i}$ as
 $H$-representations
 and $(\ell^\ast)^{\otimes
   N}$ is trivial
 by the definition of $N$
 (see Proposition \ref{prop:Reichsteinproperties}\eqref{5}), and by
 the assumption $U$
 contains all irreducible representations of $K$,
 Lemma \ref{thm:recognition} below implies that all irreducible
 representations of $H=G_y$ are contained in $\Uscr^R_y$.
\end{proof}

The following lemma was used in the proof of Proposition \ref{prop:want}\eqref{T2}, which might also be of independent interest when viewed as a recognition criterion for induced representations. 

\begin{lemma}[Recognition criterion]\label{thm:recognition}
Let $K$ be a normal subgroup of $H$ such that $H/K$ is finite and let $V$ be a representation of $H$. If $V\otimes V':=:V$ for every representation $V'$ of $H/K$, then $V:=:\Ind_{K}^H\Res^{H}_{K}V$ (see \S\ref{sec:notation}). 
\end{lemma}

\begin{proof}
$\tilde{V}:=\Ind_{K}^H\Res^H_KV\cong k[H/K]\otimes V$ as $H$-representations, where the action on the right-hand side is diagonal.
Since by the assumption $k[H/K]\otimes V:=:V$, we obtain $\tilde{V}:=:V$ as desired.
\end{proof}

\subsection{Semi-orthogonal decomposition of the Kirwan resolution}

In the next theorem we collect the results we have obtained.
\begin{theorem}\label{thm:sod}
Let $X$ be a smooth $G$-scheme such that a good quotient $\pi:X\r X\quot G$ exists. Assume furthermore that $(X,G)$ satisfies \eqref{H2}.\footnote{\eqref{H2} was imposed at the beginning of \S\ref{sec:orlovsod} and has been used throughout the section explicitly or implicitly via results in \S\ref{sec:embedding}.} 
Let $\Uscr$ be a $G$-equivariant vector bundle
on $X$.
 Assume that $\Lambda:=\pi_{s\ast} \uEnd_X(\Uscr)$ is homologically homogeneous on $X\quot G$ and that $\Uscr$ is a generator in codimension $1$ (see Definitions \ref{def:hh}, \ref{sec:hhgencodim1}).   

{}
\medskip

Let us assume that the Kirwan resolution $\fks/G$ is obtained by
performing $n$ successive Reichstein transforms and $Z_j$ is blown-up
at the $j$-th step in $X_j$.  Let $Z_{j1},\dots,Z_{jt_j}$ be
representatives for the orbits of the $G$-action on the connected
components of $Z$ and let $G_{Z_{ji}}$ be the stabilizer of $Z_{ji}$ (as a connected component).  Denote by
$\pi_{Z_{ji}}:Z_{ji}\to Z_{ji}\quot G_{ji}$ the quotient map. Let
$\Uscr_0=\Uscr$ and let $\Uscr_{i}=\Uscr_{i-1}^R$, $1\leq i\leq
n$ where $(-)^R$ is as in \eqref{eq:tweak}.
Let $\Uscr_{j,Z_{ji}}$ be the restriction of $\Uscr_j$ to $Z_{ji}$ and
set
$\Lambda_{Z_{ji}}=\pi_{Z_{ji},s,*}\uEnd_{Z_{ji}}(\Uscr_{j,Z_{ji}})$.

\medskip

There exists a semi-orthogonal decomposition
\begin{equation}
\label{eq:sod}
\langle D(\Lambda), 
D(\Lambda_{Z_{ji}})_{1\leq j\leq n, 1\leq i\leq t_j,0\leq k\leq c_{ji}-2}\rangle
\end{equation}
of $D(\fks/G)$, where $c_{ji}:=\codim(Z_{ji},X_j)$, and the terms appear in the lexicographic order (according to the label $(j,i,k)$). 
\end{theorem}

\begin{remark}\label{rem:NCCRimpliesU1}
The assumptions on $\Uscr$ and $\Lambda$ are satisfied if we assume that $\Lambda$ is an NCCR of $X\quot G$ and $X$ is \hyperlink{gnc}{``generic''}. See Proposition \ref{prop:NCCRimpliesU1}. 
\end{remark}

\begin{proof}[Proof of Theorem \ref{thm:sod}]
The theorem follows from Corollary \ref{cor:sod}, once we prove that when we perform the last {Reichstein transform} we get $D(\Lambda^{\K})\cong D(\fks/G)$. 

Assume thus that we are at the last step of the Kirwan resolution.  
We have $\Lambda^R=\pi_{s*}^R\uEnd_{\fks}(\Uscr_n)$.   
Moreover, $\fks/G$ is a smooth Deligne-Mumford stack, $\Uscr_n$ is generator in codimension $1$ by Proposition \ref{prop:want}\eqref{T2} (and the assumption on $\Uscr$), and $\Lambda^R$ is homologically homogeneous by Proposition \ref{prop:want}\eqref{T1} (and the assumption on $\Lambda$). 
Hence, by Theorem \ref{thm:saturated}, $\Uscr_n$ is full. 
Consequently, Lemma \ref{lem:full} implies that $\Qch(\fks/G)\cong \Qch(\Lambda^R)$.  
Then, $D(\Lambda^R)\cong D(\fks/G)$ as $D_{\Qch}(-)=D(\Qch(-))$ in our case by (the proof of) \cite[Theorem 1.2]{HallNeemanRydh}.  
\end{proof}
\begin{remark} The embedding $D(\Lambda)\hookrightarrow D(\fks/G)$ obtained from \eqref{eq:sod} is the same one as the one obtained from the diagonal in \eqref{diagram}. Indeed tracing through the various constructions we find that both embeddings are obtained as the composition of $D(\Lambda)\cong \langle \Uscr\rangle^\lll_{X\quot G}\subset D(X/G)$ with the pullback $D(X/G)\r D(\fks/G)$.
\end{remark}
\begin{remark}\label{rem:spoiler}
Theorem \ref{thm:sod} will not be the end of our story as we will show in \S\ref{sec:geometric} that 
the components $D(\Lambda_{Z_{ji}})$ of the semi-orthogonal decomposition \eqref{eq:sod} can be decomposed further as sums of derived categories of Azumaya algebras on smooth Deligne-Mumford stacks.
In other words the ``extra'' components to be added to the noncommutative resolution to obtain the Kirwan resolution are very close to commutative (they are ``gerby''). 
The precise statement, which is given in Corollary \ref{cor:external}, is a bit technical but it becomes very easy in the case that $G$ is abelian. In that case we have
\[
D(\Lambda_{Z_{ji}})\cong D(Z_{ji}/(G/H_{ji}))^{\oplus N_{ji}}
\]
where $H_{ji}$ is the  stabilizer of $Z_{ji}$ and $N_{ji}$ is the number of distinct $H_{ji}$-characters occurring in $\Uscr_{j,Z_{ji},x}$ for some $x\in Z_{ji}$. Thus in the abelian case the extra components are truly commutative.
\end{remark}

\subsection{A counterexample}
\label{sec:example}
Here we give an explicit example of a Cohen-Macaulay~$\Lambda$ such that~$\gldim \Lambda<\infty$, $\gldim \Lambda'=\infty$. This was announced in the beginning of \S\ref{sec:orlovsod}, from where we borrow the notations.

\begin{example}
Assume that $\Lambda$ is homologically homogeneous graded algebra and let $R$ be the center of $\Lambda$.  
We note that $(\Lambda_f)_{\geq 0}$ and $(\Lambda_f)_0$ for a homogeneous $f\in R_{>0}$ need not have finite global dimension.   As explained  in the first paragraph of the proof of Lemma \ref{lem:cm}, $\Lambda'$ is locally
of the form $(\Lambda_f)_{\ge 0}$.

For example, let $G=G_m$ act on a $4$-dimensional vector space $W$ with weights $-2,-1,1,2$. Let $U$ be another $G$-representation with weights $0,1,2$. Let $S=\Sym W^\vee$, $R=S^G$, $\Lambda=(\End(U)\otimes S)^{G}$. Then $\Lambda$ is an NCCR of $R$ \cite[Theorem 8.9]{VdB04}, and thus in particular homologically homogeneous. We let~$f$ be the product of the weight vectors in $W^\vee\subset S$ with
weights $-2,2$
 (which is $G$-invariant and thus belongs to $R$), and claim that $\gldim (\Lambda_f)_0=\infty$, which implies $\gldim (\Lambda_f)_{\geq 0}=\infty$ by \cite[Lemma 4.3.2]{SVdB}. 

Note that $B:=(\Lambda_f)_0=(\End(U)\otimes (S_f)_0)^G=M_{G,(S_f)_0}(\End U)$. By \cite[Lemma 4.2.1]{SVdB}, it is enough to show that the global dimension of $(\End(U)\otimes k[N_x])^{G_x}$ is infinite for some closed point $x\in \Spec((S_f)_0)$ with closed $G_m$-orbit and with (linear) slice $N_x$. 
To compute the slice $N_x$ we observe that $\Spec((S_f)_0)$ is an open subset in  $\PP(W)=(W\setminus \{0\})/G_m$. Hence we may compute the slice in $\PP(W)$. Let $x^*$ in $W$ be a lift of $x$ and let $N_{x^*}$ be the slice in $W$ of $x^*$ for the $G\times G_m$-action. Then
  it is easy to see that $N_x\quot G_x$ 
   is the same as $N_{x^*}\quot (G\times G_m)_{x^*}$.   
The weights of $W$, $U$ as $G\times G_m$-representation are respectively $(-2,1),(-1,1),(1,1),(2,1)$ and $(0,0),(1,0),(2,0)$. 
We take the point $x=[a:0:0:b]\in \Spec (S_f)_0\subset \PP(W)$. The stabilizer of $x^*$ is $\ZZ_4$, embedded in $G\times G_m$ via $(\epsilon,\epsilon^2)$, where $\epsilon$ is a primitive $4$-th root of unity. The actions of $\ZZ_4$ on $N_{x^*}$, $U$ have weights $1/4,3/4$ and $0,1/4,2/4$, respectively. 
Thus, $N_x\quot G_x\cong N_{x^*}\quot \ZZ_4$ is a Gorenstein singularity. 
We moreover have  $(\End(U)\otimes k[N_x])^{G_x}= (\End(U)\otimes k[N_{x^*}])^{\ZZ_4}=\End_{k[N_{x^*}]^{\ZZ_4}}((U\otimes k[N_{x^*}])^{\ZZ_4})$, where the last equality follows by \cite[Lemma 4.1.3]{SVdB} because $N_{x^*}$ is a \hyperlink{gnc}{generic}  $\ZZ_4$-representation. 
If $B$ would  have finite global dimension it would be an NCCR. However this is impossible  by \cite[Proposition 4.5]{IW} since 
$(U\otimes k[N_{x^*}])^{\ZZ_4}$ is not an ``MM-module'' (see loc.cit.) as $U\subset k\ZZ_4$ (as $\ZZ_4$-representation) and $\End_{k[N_{x^*}]^{\ZZ_4}}(k\ZZ_4\otimes k[N_{x^*}])^{\ZZ_4})$ is a Cohen-Macaulay $k[N_{x^*}]^{\ZZ_4}$-module.
\end{example}

\section{Endomorphism sheaves in the case of constant stabilizer dimension}\label{sec:geometric}
In this section, on smooth quotients stacks with constant stabilizer
dimension, we give a geometric (``gerby'') interpretation of sheaves
of endomorphism algebras of vector bundles. In particular, this
applies to $\Lambda_{Z_{ji}}$ appearing in Theorem \ref{thm:sod}.

\subsection{Normalizer of a representation}
We discuss some technical results we need later on. Let $H\subset G$ be an inclusion of reductive groups. We recall the following result for further reference. 
\begin{lemma}\cite[Lemma 1.1]{LunaRichardson}\label{lem:N(H)reductive}
$N(H)$ is reductive.
\end{lemma}

 Let $V$ be an irreducible representation of 
$H$. Let $g\in N(H)$. Denote by $\sigma_g=g^{-1}\cdot g:H\to H$ and by ${}_{\sigma_{g}}V$  the corresponding twisted $H$-representation (i.e. the action of $h\in H$ on ${}_{\sigma_g}V_i$ is
  $h.v:=(g^{-1}hg)v$). 
We set 
\[
N_V(H):=\{u\in N(H)\mid {}_{\sigma_{u}}\!V\underset{H}{\cong} V\}
\]
so that we have inclusions
\[
H\subset N_V(H)\subset N(H).
\]
\begin{lemma}\label{lem:finiteisoV} 
The  index of $N_V(H)$ in $N(H)$ is finite. 
\end{lemma}
\begin{proof} 
We  claim that if a reductive group $H$ is a normal subgroup in a reductive group $K$, then the image of the map 
\begin{equation}
\label{eq:outmap}
K\to \operatorname{Out}(H)=\operatorname{Aut}(H)/\operatorname{Inn}(H),\quad k\mapsto (h\mapsto khk^{-1})
\end{equation}
is finite. We apply this with $K=N(H)$, which is reductive by Lemma \ref{lem:N(H)reductive}. If $u\in N(H)$ is in the kernel of \eqref{eq:outmap} then $\sigma_u$ is an inner
automorphism of $H$ and then ${}_{\sigma_u}\! V\cong V$. Hence $u\in N_V(H)$. So the kernel of \eqref{eq:outmap} is contained in $N_V(H)$, which is therefore of finite index.

\medskip

We now prove the claim.
Note that we can assume that $H$ is connected. Indeed $H_e$ is a normal subgroup of $K$ and furthermore $\operatorname{Out}(H)\r \Aut(H_e)/\operatorname{Inn}(H)$ has finite kernel (since the kernel is a subquotient of $\Aut(H/H_e)$ which is finite as $H/H_e$ is finite), and $\Aut(H_e)/\operatorname{Inn}(H)$ is a quotient of $\operatorname{Out}(H_e)$.

Assuming $H$ connected we have $H\subset K_e$. As $K/K_e$ is finite we may then also assume that $K$ is connected.
Then $K=HQ$ for a subgroup $Q$ of $K$ such that $Q$ and $H$ commute \cite[Theorem 8.1.5, Corollary 8.1.6]{Springer}. Thus, the image of $K$ is trivial in this case.
\end{proof}
For use below we write  $V\sim V'$ for $V,V'\in \rep(H)$ if  there is some $g\in N(H)$ such that $V'\overset{H}{\cong} {}_{\sigma_g} V$. This defines an equivalence relation
on $\rep(H)$ and the equivalence classes are in bijection with $N(H)/N_V(H)$. In particular
 by Lemma \ref{lem:finiteisoV} they are finite.
\subsection{Actions with stabilizers of constant dimension}\label{sec:actstab}
Now we assume that
 ~$\XZ$ is a $G$-equivariant connected\footnote{The connectedness assumption is purely to simplify the notation. It is not a serious restriction as in general $X/G\cong \coprod_i X_i/G_i$ where
the $X_i$ are representatives of the orbits of the connected components of $X$ and the $G_i$ are their stabilizers.} smooth
$k$-scheme with a good quotient $\pi:\XZ\r \XZ\quot G$.
Moreover we assume that the stabilizers $(G_x)_{x\in \XZ}$ have
dimension independent of~$x$. As explained in \S\ref{sec:conststab}
all orbits in $\XZ$ are closed and all $G_x$
are reductive.

Let $H$ be the stabilizer of a point in the open (``principal'') stratum of the Luna stratification \cite{Luna} (we call $H$ a {\em generic stabilizer}).
By the properties of the Luna stratification, $H$ is uniquely determined up to conjugacy. 
\begin{proposition}\label{prop:structure}
Let $\XZ^{\langle H\rangle}$ be the union of connected components of $\XZ^H$ which contain a point whose stabilizer is exactly $H$.  Then $\XZH$ is smooth and the canonical map  
\[
\phi: G\times^{N(H)} \XZ^{\langle H\rangle} \overset{\cong}{\to} \XZ
\]
is an isomorphism.
\end{proposition}

\begin{proof}
Since $\XZ^H$ is smooth \cite[Proposition A.8.10(2)]{pseudoreductive},  $\XZH$ is smooth. 
  By \cite{LunaRichardson}, \cite[Theorem 7.14]{PopovVinberg},
  $\XZ^{\langle H\rangle}\quot N(H)\cong \XZ\quot G$. It thus follows that
  $\phi$ is surjective since $\XZ\r \XZ\quot G$,
  $\XZ^{\langle H\rangle}\r \XZ^{\langle H\rangle} \quot N(H)$ separate
  orbits (as all orbits are closed as mentioned in the beginning of
  this subsection). Moreover $\phi$ defines an isomorphism  between the principal strata for the Luna
  stratification. Globally $\phi$ is 
  quasi-finite since $G$ acts with constant stabilizer dimension. As $\XZ$ is normal, by Zariski's main theorem $\phi$ is an isomorphism.
\end{proof}
\begin{remark} Assume that $(\XZ,\Lscr)$ is a linearized connected smooth $G$-scheme. A point $x\in \XZ^{\operatorname{ss}}\coloneqq X^{\operatorname{ss},\Lscr}$ is \emph{stable in Mumford's sense} \cite{Mumford} if 
  $Gx$ has maximal dimension and is  closed in $\XZ^{\operatorname{ss}}$. Let $\XZ^{\operatorname{ms}}\subset \XZ^{\operatorname{ss}}$ be the set of Mumford stable points. Proposition~\ref{prop:structure}
applies to $\XZ^{\operatorname{ms}}$ and so gives a structure theorem for $\XZ^{\operatorname{ms}}$. We have not been able to find this result in the literature.
\end{remark}
For use below we introduce some associated notations. For $V\in \rep(H)$ we put  $\XscrZ_V\coloneqq \XZ^{\langle H\rangle}/(N_V(H)/H)$.
 For convenience
we list some easily verified properties of $\XscrZ_V$.
\begin{itemize}
\item
$\XscrZ_V$ is a smooth Deligne-Mumford stack.
\item
The natural quotient map $\XscrZ_V\r \XZ^{\langle H\rangle }\quot (N(H)/H)\cong \XZ\quot G$ is finite.
\item $\XscrZ_V$ may however be non-connected.
\item If $G$ is abelian then $\XscrZ_V=\XZ/(G/H)$, independently of $V$.
\end{itemize}

\subsection{Equivariant vector bundles and Azumaya algebras}
In this section we assume as in \S\ref{sec:actstab} that  $\XZ$ has constant
stabilizer dimension. 
We discuss some properties of equivariant vector bundles on $\XZ$. For simplicity we will phrase them for a fixed choice of $H$ (within its conjugacy class) but it is easy to see
that they are in fact independent of this choice.

If $U$ is an $N(H)$-representation (possibly infinite dimensional) and $V$ is an irreducible $H$-representation
then we let $U(V)$ be the $V$-isotypical part of $U$; i.e. if $U_V\coloneqq \Hom(V,U)^H$ then
$U(V)$ is the image of the
evaluation map
\begin{equation*}
\label{eq:evaluation}
V\otimes U_V\r U.
\end{equation*}
The evaluation map is injective so it yields in particular an isomorphism as $H$-representations
\begin{equation}
\label{eq:evaliso}
V\otimes U_V\cong U(V)
\end{equation}
where the $H$-action on $U_V$ is trivial.
Moreover there is an internal 
direct sum decomposition
\begin{equation}
\label{eq:internal}
U=\bigoplus_{V\in \rep(H)} U(V).
\end{equation}
One checks that if $g\in N(H)$ then inside $U$
\begin{equation}
\label{eq:blockaction}
g(U(V))=U({}_{\sigma_g} V).
\end{equation}
It follows that $U(V)$ is in fact a $N_V(H)$-subrepresentation of $U$.

\medskip

Let  $\Uscr$ be a $G$-equivariant vector bundle on $\XZ$. We write
\[
\Uscr^{\langle H\rangle}:=\Uscr\mid \XZ^{\langle H\rangle}.
\]
Since $\Uscr^{\langle H\rangle}$, locally over $\XZ^{\langle H\rangle}\quot N(H)$, is in particular an $N(H)$-representation it makes sense to use the notation
$
\Uscr^{\langle H\rangle}(V)
$.

We will also put
\[
\Uscr^{\langle H\rangle}_V:=\Hom(V,\Uscr^{\langle H\rangle})^H.
\]
From \eqref{eq:evaliso} (checking locally over $\XZ^{\langle H\rangle}\quot N(H)$) we get
\begin{equation}
\label{eq:triv1}
\Uscr^{\langle H\rangle}(V)\cong V\otimes \Uscr_V^{\langle H\rangle}
\end{equation}
as $H$-equivariant coherent sheaves on $\XZ^{\langle H\rangle}$.
\begin{lemma}\label{lem:uuv}  $\Uscr^{\langle H\rangle}(V)$ and $\Uscr_V^{\langle H\rangle}$ are 
vector bundles on $\XZ^{\langle H\rangle}$. Moreover,  $\Uscr^{\langle H\rangle}(V)$ is  a $N_V(H)$-equivariant subbundle of $\Uscr^{\langle H\rangle}$.
\end{lemma}
\begin{proof}  By \eqref{eq:triv1} it suffices to consider $\Uscr^{\langle H\rangle}(V)$ for the first claim.
We have a decomposition of coherent sheaves on $\XZ^{\langle H\rangle}$:
\begin{equation}\label{eq:uh}
\Uscr^{\langle H\rangle}=\bigoplus_{V\in \rep(H)} \Uscr^{\langle H\rangle}(V).
\end{equation}
This can be checked locally over $\XZ^{\langle H\rangle}\quot N(H)$ so that we may assume
that $\XZ^{\langle H\rangle}$  is affine. Then it follows from  \eqref{eq:internal}.
In particular $\Uscr^{\langle H\rangle}(V)$ is a direct summand of $\Uscr^{\langle H\rangle}$. So
it is a vector bundle.

The fact that $\Uscr^{\langle H\rangle}(V)$ is $N_V(H)$-equivariant may again be checked in the case
that $\XZ$ is affine where it follows from the above discussed fact that $U(V)$ is $N_V(H)$-equivariant.
\end{proof}
\begin{lemma} \label{lem:nonzero}
Let $x\in \XZ^{\langle H\rangle}$. Then
$\Uscr^{\langle H\rangle}(V)\neq 0$ (which is equivalent to $\Uscr_V^{\langle H\rangle}\neq 0$) if and only if 
there exists $V'\sim V$ such that $V'$ appears in $\Uscr^{\langle H\rangle}_x$.
\end{lemma}
\begin{proof} We first collect some easy facts. From \eqref{eq:blockaction} one may deduce that if $y\in \XZ^{\langle H\rangle}$ and $g\in N(H)$ then
\begin{equation}
\label{eq:easy1}
\Uscr^{\langle H\rangle}(V)_{g^{-1} y}\overset{H}{\cong} \Uscr^{\langle H\rangle}({}_{\sigma_g}V)_{y}.
\end{equation}
Moreover if $y,y'$ are in the same connected component of $\XZ^{\langle H\rangle}$ then by semi-continuity
\begin{equation}
\label{eq:easy2}
\Uscr^{\langle H\rangle}(V)_{y'}\overset{H}{\cong}\Uscr^{\langle H\rangle}(V)_{y}.
\end{equation}
Now we prove the lemma.
\begin{itemize}
\item[($\Rightarrow$)] If $\Uscr^{\langle H\rangle}(V)\neq 0$ then there is some $y\in \XZ^{\langle H\rangle}$
such that  $\Uscr^{\langle H\rangle}(V)_y\neq 0$. $y$ may be in a different component than $x$, but by
combining \eqref{eq:easy1}\eqref{eq:easy2} we find that there exists $V'\sim V$ such that
$\Uscr^{\langle H\rangle}(V')_x\neq 0$ (as $N(H)$ acts transitively on the connected components by using the assumption that $\XZ$ is connected).
\item[($\Leftarrow$)] 
This is proved by reversing the argument in  ($\Rightarrow$).
\qedhere
\end{itemize}
\end{proof}
For use below we will write $\langle \Uscr\rangle \subset \rep(H)/\sim$ for the set of equivalence
classes that contain a representation that appears in $\Uscr_x^{\langle H\rangle}$  for some $x\in \XZ^{\langle H\rangle}$.
It follows from \eqref{eq:easy1} 
that $\langle \Uscr\rangle$ is well-defined.

\medskip

It follows from Lemma \ref{lem:uuv} 
that
\begin{align*}
\Ascr_V&\coloneqq\uEnd_{\XZ^{\langle H\rangle}}(\Uscr^{\langle H\rangle}(V))^H\\
&\overset{\eqref{eq:triv1}}{\cong} \uEnd_{\XZ^{\langle H\rangle}}(\Uscr^{\langle H\rangle}_V)
\end{align*}
is a $N_V(H)/H$-equivariant sheaf of Azumaya algebras on $\XZ^{\langle H\rangle}$ which is trivial
if we forget the $N_V(H)/H$-action. Below we consider $\Ascr_V$ as living on the quotient stack $\XscrZ_V=\XZ^{\langle H\rangle}/(N_V(H)/H)$ which was introduced in \S\ref{sec:actstab}. 
Our main result in this section is the following.
\begin{proposition} \label{prop:mainpropZ}
Assume that $\Uscr$ is a saturated $G$-equivariant vector bundle on~$\XZ$. Put
\[
\Lambda=\pi_{s\ast} \uEnd_\XZ(\Uscr)
\]
where $\pi_s:\XZ/G\r \XZ\quot G$ is the quotient map.
Then we have an equivalence
of abelian categories
\begin{equation}
\label{eq:decomp}
\Qch(\XZ\quot G,\Lambda)\cong \bigoplus_{V\in \langle \Uscr\rangle/{\sim}} \Qch(\XscrZ_V,\Ascr_V).
\end{equation}
If $G$ is abelian then each class in $\rep(H)/{\sim}$ is a singleton and  
$\Qch(\XscrZ_V,\Ascr_V)\cong \Qch(\XZ/(G/H))$ for $\{V\}\in \langle \Uscr\rangle$.
\end{proposition}
\begin{proof} The part about general $G$ 
follows by combining Lemmas \ref{lem:decomp} and \ref{lem:morita} below where we use Lemma \ref{lem:nonzero} to restrict the sum.

Let us now assume that $G$ is abelian. Then $\XZ^{\langle H\rangle}=\XZ$ and we may drop $(-)^{\langle H\rangle}$ superscripts. It is obvious that every class in $\rep(H)/{\sim}$ is a singleton.
Furthermore we may extend the $H$-action on $V$ to a $G$-action (non-canonically). It follows that $\Uscr_V$ is $G/H$-equivariant and $\Ascr_V=\uEnd_\XZ(\Uscr_V)$ as $G/H$-equivariant sheaves
of algebras. In other words $\Ascr_V$ is a trivial Azumaya algebra on $\XZ/(G/H)$ and the result follows. 
\end{proof}
\begin{corollary}\label{cor:corZ} With notations and hypotheses as in Proposition \ref{prop:mainpropZ} we have a decomposition 
\begin{equation}
\label{eq:decomp1}
D(\Lambda)\cong \bigoplus_{V\in \langle \Uscr\rangle/{\sim}} D(\Ascr_V).
\end{equation}
\end{corollary}

\begin{proof}
We only need to note that in $D_{\Qch}(-)=D(\Qch(-))$ in our case by (the proof of) \cite[Theorem 1.2]{HallNeemanRydh}.
\end{proof}

\subsection{Geometric interpretation of $\Lambda_{Z_{ji}}$}

In particular, Proposition \ref{prop:mainpropZ} and Corollary \ref{cor:corZ} apply to the setting of Theorem \ref{thm:sod}, and thus allow us to give a more geometric description of $\Lambda_{Z_{ji}}$ appearing there. For the convenience of the reader we repeat the statements in that setting. 

We use the notation introduced in Theorem \ref{thm:sod}, moreover we set $H_{ji}$ for the principal stabilizer of the action of $G_{ji}$ on $Z_{ji}$, 
$\Ascr_{Z_{ji},V}=\uEnd_{Z^{\langle H_{ji}\rangle}}(\Uscr_{j,Z_{ji},V}^{\langle H_{ji}\rangle})$ $\Zscr_{ji,V}=Z_{ji}^{\langle H_{ji}\rangle}/(N_V(H_{ji})/H_{ji})$.

\begin{corollary} \label{cor:external}
Let the setting be as in Theorem \ref{thm:sod}. Then
\begin{align*}
\Qch(\XZ_{ji}\quot G_{ji},\Lambda_{Z_{ji}})&\cong \bigoplus_{V\in \langle \Uscr_{j,Z_{ji}}\rangle/{\sim}} \Qch(\XscrZ_{ji,V},\Ascr_{Z_{ji},V}),\\
D(\Lambda_{Z_{ji}})&\cong \bigoplus_{V\in \langle \Uscr_{j,Z_{ji}}\rangle/{\sim}} D(\Ascr_{Z_{ji},V}).
\end{align*}
\end{corollary}
\begin{proof} 
We need to check that the hypotheses for Proposition \ref{prop:mainpropZ} with  $(Z,G,\Uscr)=(Z_{ji},G_{ji},\Uscr_{j,Z_{ji}})$  apply. 

Recall that $Z_{ji}$ is $G_{ji}$-equivariant smooth connected $k$-scheme, and by definition $Z_{ji}$ has stabilizers of constant dimension. Let us denote $\Uscr_{ji}:=\Uscr_{j,Z_{ji}}$. 
We only need to observe that the hypothesis on $\Uscr$ imply that $\Uscr_{ji}$ is saturated. This follows by Theorem \ref{thm:saturated} as its  hypotheses are satisfied by Proposition \ref{prop:want} (and the initial hypothesis on $\Lambda$, $\Uscr$).
\end{proof}

\subsection{A decomposition result}
In this section we assume as in \S\ref{sec:actstab} that  $\XZ$ has constant
stabilizer dimension. We keep the notations introduced in the previous sections.
\begin{lemma} \label{lem:decomp}
For $V\in \rep(H)$ consider the morphism
\[
\psi_V: \XZ^{\langle H\rangle}\quot (N_V(H)/H)\r \XZ^{\langle H\rangle}\quot (N(H)/H)\cong \XZ\quot G.
\]
Let $\Uscr$ be a $G$-equivariant vector bundle on $Z$. Then we have
\[
\pi_{s\ast}\uEnd_\XZ(\Uscr)\cong\bigoplus_{V\in \rep(H)/\sim} \psi_{V,\ast}\pi_{V,s,\ast} \Ascr_V
\]
where $\pi_s:\XZ/G\r \XZ\quot G$, $\pi_{V,s}:\XscrZ_V=\XZ^{\langle H\rangle}/(N_V(H)/H)
\to \XZ^{\langle H\rangle}\quot (N_V(H)/H)$ are the quotient maps.
\end{lemma}
\begin{proof} We consider the corresponding quotient map
\[
\pi^{\langle H\rangle}_{s\ast}
:\XZ^{\langle H\rangle}/N(H)\r \XZ^{\langle H\rangle}\quot N(H)
\cong \XZ\quot G.
\]
Using $\XZ\quot G=\XZ^{\langle H\rangle}/N(H)$ we obtain
\[
\pi_{s\ast}\uEnd_\XZ(\Uscr)=\pi^{\langle H\rangle}_{s\ast}\uEnd_{\XZ^{\langle H\rangle}}(\Uscr^{\langle H\rangle}).
\]
We may now restrict to the case\footnote{Note that $\XZ^{\langle H\rangle}$ may be nonconnected. So we are stepping out of our original context. However we will be careful not to use any
results depending on connectedness.} $\XZ=\XZ^{\langle H\rangle}$, $G=N(H)$. We drop all superscripts 
$(-)^{\langle H\rangle}$ from the notation.

Using \eqref{eq:uh} and Lemma \ref{lem:uuv} we obtain a $G$-equivariant decomposition of $\Uscr$,
\[
\Uscr=\bigoplus_{V\in \rep(H)/\sim} \bigoplus_{V'\sim V} \Uscr(V')
\]
so that 
\[
\pi_{s\ast} \uEnd_\XZ(\Uscr)=\bigoplus_{V\in \rep(H)/\sim} \pi'_{s\ast}\left(\bigoplus_{V'\sim V}\uEnd_{\XZ}(\Uscr(V'))^H\right)
\]
where $\pi'_{s\ast}$ is the modified quotient map
\[
\XZ/(G/H)\r \XZ\quot (G/H).
\]

Using the definition of $\Ascr_V$ it is now sufficient to prove that the projection
\[
\pi'_{s\ast}\left(\bigoplus_{V'\sim V}\uEnd_{\XZ}(\Uscr(V'))^H\right)\r \psi_{V,\ast}\pi_{V,s,\ast} \uEnd_{\XZ}(\Uscr(V))^H
\]
is an isomorphism. This can be checked locally over $\XZ\quot (G/H)$ and hence we may assume that $\XZ$ is affine. Then 
it reduces to the algebraic statement in Lemma \ref{lem:algebraic} below (with $G=G/H$, $K=N_V(H)/H$).
\end{proof}
\begin{lemma}
\label{lem:algebraic} Let $K\subset G$ be  groups. Let $A=\bigoplus_{u\in G/K} A_u$ be an algebra
equipped with a $G$-action such that $g(A_k)=A_{gk}$. Then projection induces an algebra isomorphism
\[
\bigl(\bigoplus_{u\in G/K} A_u\bigr)^G\r A_e^{K}.
\]
\end{lemma}
\begin{proof} Left to the reader.
\end{proof}
\subsection{Morita theory of $\Ascr_V$}
In this section we assume as in \S\ref{sec:actstab} that  $\XZ$ has constant
stabilizer dimension. We keep the notations introduced in the previous section.
Consider the quotient map
\[
\pi_{V,s}:\XscrZ_V\r \XZ^{\langle H\rangle} \quot (N_V(H)/H)=\XZ^{\langle H\rangle}\quot N_V(H)
\]
as well as the associated morphism of ringed stacks
\[
\bar{\pi}_{V,s}:(\XscrZ_V,\Ascr_V)\r (\XZ^{\langle H\rangle} \quot N_V(H),\pi_{V,s,\ast}\Ascr_V).
\]
\begin{lemma} \label{lem:morita}
Assume that $\Uscr$ is a $G$-equivariant saturated vector bundle on $Z$. Then for every $V\in \rep(H)$ there is an equivalence of categories
\[
\bar{\pi}_{V,s,\ast}:\Qch(\XscrZ_V,\Ascr_V)\r \Qch(\XZ^{\langle H\rangle} \quot N_V(H),\pi_{V,s,\ast}\Ascr_V).
\]
\end{lemma}
\begin{proof} 
To simplify the notation we first replace $\XZ$ by $\XZ^{\langle H\rangle}$ and $G$ by $N(H)$ and drop all $(-)^{\langle H\rangle}$ superscripts. 
Since $X/G\cong X^{\langle H\rangle}/N(H)$
it is easy to see that this does not affect the saturation property of $\Uscr$. 

Next we further replace $G$ by $N_V(H)$ which by Lemma \ref{lem:saturatedres} below also does not affect the saturation property.

As $\bar{\pi}_{V,s,\ast}\bar{\pi}_{V,s}^\ast$ is easily seen to be the identity, we have 
to prove that $\bar{\pi}_{V,s}^\ast\bar{\pi}_{V,s,\ast}$ is the identity. 

This may be checked strongly \'etale locally on $\XZ$. Hence
we may replace $\XZ$ by $G\times^{G_x} S$ for $S$ a smooth connected affine slice at $x\in \XZ$. Using $(G\times^{G_x} S)/(G/H)\cong S/(G_x/H)$
we may reduce do $\XZ=S$, $G=G_x$; i.e. $x\in \XZ$ is now a fixed point for~$G$ and we have to show that $\bar{\pi}_{V,s}^\ast\bar{\pi}_{V,s,\ast}$ is
the identity on a neighborhood of $x$. Note that $G/H$ is now a finite group.
 
Since $\bar{\pi}_{V,s,\ast}$ is exact and $\bar{\pi}^\ast_{V,s}$ is right exact it is sufficient to prove that for every $\Mscr\in \coh(\Ascr_V)$ there
is a map $\Ascr_V^{\oplus N}\r \Mscr$ in $\coh(\Ascr_V)$, whose cokernel is zero on a neighborhood of $x$. By lifting generators of $\Mscr$
we may reduce to the case $\XZ=x$ and we have to show that $\Ascr_V$ is a projective generator for $\coh(\Ascr_V)$. 
As $\Uscr$ is saturated 
this follows from Lemma \ref{lem:morita1} below (using that $H_x=H$).  
\end{proof}
\begin{lemma}\label{lem:saturatedres}
Assume that $\Uscr$ is a saturated $G$-equivariant vector bundle on $Z$ and~$K$ is a subgroup of $G$ of finite index which contains
  $H_x$ for all $x\in H$. Then the  pullback of $\Uscr$ to $Z/K$
  is also saturated. 
\end{lemma}
\begin{proof} Let $x\in Z$. As $G/K$ is finite we have $T_x(X)/T_x(Kx)=T_x(X)/T_x(Gx)=H_x$.
As $G_x/K_x$ is finite, the lemma follows from Mackey's restriction formula.
\end{proof}

\begin{lemma}\label{lem:morita1}
Let $H$ be a normal subgroup of finite index in $G$. Assume that $U$ is a finite dimensional $G$-representation which is up to nonzero multiplicities induced
from $H$. If for $V\in \rep(H)$, ${}_{\sigma_g} V\cong V$ for all $g\in G$, then
$\End_H(U(V))$ is a projective generator for  $\mod(G/H,\End_H(U(V))$.
\end{lemma}

\begin{proof}
Note that 
$\Res^G_H\Ind_H^G W:=:\bigoplus_{W'\sim W} W'$ for $W\in \rep(H)$ (see e.g. the proof of \cite[Lemma 4.5.1]{SVdB}). 
Hence up to Morita equivalence we may assume $U=\bigoplus_{i=1}^n \Ind^G_H V_i$ with $V_i\in \rep(H)$ and $V_1=V$, $V_2,\ldots, V_n\not\sim V$ (as $\Ind_H^GW\cong \Ind_H^GW'$ if $W\sim W'$), so that $\Ind^G_H V_1$ and $\bigoplus_{i=2}^n \Ind^G_H V_i$ have no common
$H$-summands. Thus $U(V)=\Ind^G_HV$ (as ${}_{\sigma_g} V\cong V$). We now put $U=U(V)$.

As $G$-representations we have
\begin{equation}
\label{eq:step1}
U\otimes k[G/H]\cong \Ind^G_HV\otimes k[G/H]=\Ind^G_H(V\otimes k[G/H]):=:\Ind^G_H V=U
\end{equation}
where we used the projection formula (i.e. the tensor identity, \cite[Proposition I.3.6]{Jantzen}) for the second equality, and the fact $V:=:V\otimes k[G/H]$ as $H$-representations 
(since $k[G/H]$ is the trivial $H$-representation) for the $:=:$-relation.

Consider the  functor
\[
F:\mod(G)\r \mod(G/H,\End_H(U)):M\mapsto \Hom_H(U,M).
\]
One checks that if $T\in \mod(G)$, $W\in \mod(G/H)$ then
\begin{equation}
\label{eq:step2}
F(T\otimes W)=F(T)\otimes_k W
\end{equation}
in $\mod(G/H,\End_H(U))$. Applying $F$ to \eqref{eq:step1} and using \eqref{eq:step2} with $W=k[G/H]$ we obtain
\begin{equation}
\label{eq:step3}
\End_H(U)\otimes k[G/H] :=:\End_H(U)
\end{equation}
in $\mod(G/H,\End_H(U))$.
As $\End_H(U)\otimes k[G/H]$ is tautologically a generator for $\mod(G/H,\End_H(U))$ it follows from \eqref{eq:step3}
that $\End_H(U)$ is a generator for $\mod(G/H,\End_H(U))$.
\end{proof}

\section{Example}\label{sec:conifold}
We demonstrate the above results on a simple example of  the conifold singularity.

\medskip

Assume that $X=W$ is a $4$-dimensional vector space on which $G=G_m$ acts with weights $-1,-1,1,1$. Then $X\quot G$ is a conifold singularity. 
In this case $Z$ is the origin, and the Kirwan resolution $\fks /G$ is obtained by one Reichstein transform. 

A noncommutative crepant resolution $\Lambda$ of $X\quot G$, is given by a vector bundle $\Uscr=\Oscr_X\oplus \chi_1\otimes\Oscr_X$, where $\chi_i$ denotes $1$-dimensional $G$-representation with weight $i$, i.e. $\Lambda=(\End(\chi_0\oplus \chi_1)\otimes k[W^\vee])^G\cong \End_{k[W^\vee]^G}(k[W^\vee]^G\oplus (\chi_1\otimes k[W^\vee])^G)$ \cite[Theorem 8.9]{VdB04}. 

Note that $\Lambda_Z=\End(\chi_0\oplus \chi_1)^G=k^{\oplus 2}$ and $\codim(Z,X)=4$. 
By Theorem \ref{thm:sod}
 we then obtain 
\begin{equation}\label{eq:conifold}
D(\fks/G)=\langle D(\Lambda),D(k),D(k),D(k),D(k),D(k),D(k)\rangle.
\end{equation} 

\begin{remark}
Note that $X\quot G$ is as a toric variety given by a fan with a single cone $\sigma$ generated by $(0,0,1),(1,0,1),(1,1,1),(0,1,1)$. 
Let $\Sigma=\sigma\cup \RR_{\geq 0}(1,1,2)$.
Then $\fks/G$ is a toric stack  given by the stacky fan 
\[\mathbf{\Sigma}=(\Sigma,\{(0,0,1),(1,0,1),(1,1,1),(0,1,1),(2,2,4)\})\] (representing the stacky blow up of $X\quot G$ in the origin) and $\fks\quot G$  is a toric variety given by $\Sigma$ (which is a blow-up of $X\quot G$ in the origin) \cite[Theorem 4.7]{EdidinMore}.
Using \cite[Proposition 4.5]{borisov2005orbifold} one can (alternatively) check that $\rk K_0(\fks/G)=8$, which agrees with \eqref{eq:conifold}.

The blow-up of $X\quot G$ in the origin can also be viewed as $\Oscr(-1,-1)_{\PP^1\times \PP^1}$. From this perspective $\fks/G$ is the square-root stack of  $\Oscr(-1,-1)_{\PP^1\times \PP^1}$ along the exceptional divisor. 
\end{remark}

\appendix

\section{Local duality for graded rings}
Results in this appendix  have in a closely related form appeared in \cite{VyasYekutieli}, \cite{yekutieli2019derived}. 
However, our setting is slightly more general from the one in loc. cit. (as  $R_0$ below can be infinite dimensional), therefore for convenience of the reader we provide self-contained proofs.

Let $\Lambda$ be an $\NN$-graded ring which is finitely generated as a
module over its center $R$ which in turn is a $\NN$-graded $k$-algebra
such that $R_n$ is a finitely generated $R_0$-module. We assume that $R_0$ is 
a finitely generated $k$-algebra. 
For convenience reasons we use {\em left} modules in this appendix.
This allows us to literally use some results from
\cite{van1997existence}. Needless to say this is only a notational issue and moreover in the rest of the paper we only use Corollary
\ref{cor:dualizing} which is left right agnostic.

Below we write $D(\Lambda)$ for $D(\Gr\Lambda)$
and this convention extends to all related notations. 
 Let $D_{R}$, $D_{R_0}$ be the Grothendieck dualizing complexes of
 $R$, $R_0$ respectively and let $D_{\Lambda}$, $D_{\Lambda_0}$ be the
 corresponding dualizing complexes of $\Lambda$, $\Lambda_0$; i.e. we have
\begin{align*}
D_{\Lambda}&=\RHom_{R}(\Lambda,D_R),\\
D_{\Lambda_0}&=\RHom_{R_0}(\Lambda_0,D_{R_0}).
\end{align*}

Let $C$ be an arbitrary graded $k$-algebra. For $M$ a complex of left  graded $C\otimes\Sigma$-modules, where $\Sigma\in\{\Lambda,\Lambda_0\}$,  we put
\[
M^\vee=\RHom_{R_0}(M,D_{R_0})\in D(C^\circ\otimes \Sigma^{\circ}).
\]

Let $D_f(\Lambda_0)$ be the full subcategory of $D(\Lambda_0)$ consisting of complexes which have finitely generated cohomology.
Then $(-)^\vee$ defines a duality between 
$D_f(\Lambda_0)$ and $D_f(\Lambda_0^\circ)$ (recall that $D_{R_0}$ is bounded and has finite injective dimension).

Similarly, let $D_f^b(\Lambda)$ be the full subcategory of $D^b(\Lambda)$ consisting of complexes with cohomology which
is finitely generated as $\Lambda_0$-module (or equivalently as $R_0$-module) in every degree. Then $(-)^\vee$ defines a duality between 
$D_f^b(\Lambda)$ and $D_f^b(\Lambda^\circ)$ (recall that $D_{R_0}$ has finite injective dimension).

\begin{remark} 
If $M\in D(\Lambda_0)$ then 
by change of rings we have
\[
M^\vee=\RHom_{\Lambda_0}(M,D_{\Lambda_0}).
\]
The same formula holds if $M\in D(\Lambda)$, but 
unfortunately in that case  the formula obscures the $\Lambda$-action on $M^\vee$.
\end{remark}

Let $R\Gamma_{\Lambda_{\ge 1}}$ be the right derived functor of $\varinjlim_n\RHom_\Lambda(\Lambda/\Lambda_{\geq n},-)$.

\begin{proposition}\label{prop:localduality}
Let $M\in D(C\otimes\Lambda)$. Then there is a local duality formula in $D(C^\circ\otimes\Lambda^\circ)$
\[
R\Gamma_{\Lambda_{>0}}(M)^\vee\cong
\RHom_\Lambda(M,R\Gamma_{\Lambda_{>0}}(\Lambda)^\vee).
\]
\end{proposition}

\begin{proof}
If $C$ is absent then the formula is true on the level of complexes
for $M=\Lambda$. 
One then proceeds as in the proof of \cite[Proposition 5.1]{van1997existence} by replacing $M$ by a free  $C\otimes \Lambda$-resolution. 
\end{proof}

\begin{corollary}\label{cor:dualizing}
We have $D_\Lambda\cong R\Gamma_{\Lambda_{>0}}(\Lambda)^\vee$ in $D(\Lambda^e)$.
\end{corollary}

\begin{proof}
We first check that the right-hand side is a dualizing complex.
Note that $R\Gamma_{\Lambda_{>0}}(\Lambda)^\vee=R\Gamma_{R_{>0}}(\Lambda)^\vee$ so that we do not need to worry about the distinction between left and right. 
 Following the proof of \cite[Theorem 6.3]{van1997existence} we only need to check that $H^i(R\Gamma_{\Lambda_{>0}}(\Lambda)^\vee)$ is finitely generated. Following Lemma \ref{lem:fingen} below it is enough to verify that $\Lambda_0\Lotimes_\Lambda H^i(R\Gamma_{\Lambda_{>0}}(\Lambda)^\vee)$ has finitely generated cohomology as $\Lambda_0$-modules.  
We have the following formula as right $\Lambda_0$-modules
\begin{align*}
\RHom_{\Lambda_0^\circ}(\Lambda_0\Lotimes_\Lambda R\Gamma_{\Lambda_{>0}}(\Lambda)^\vee,D_{\Lambda_0})&
=\RHom_{\Lambda^\circ}(\Lambda_0,\RHom_{\Lambda_0}(R\Gamma_{\Lambda_{>0}}(\Lambda)^\vee,D_{\Lambda_0}))\\
&=\RHom_{\Lambda^\circ}(\Lambda_0,R\Gamma_{\Lambda_{>0}}(\Lambda))\\
&=\RHom_{\Lambda^\circ}(\Lambda_0,\Lambda)
\end{align*}
where in the first line we have considered $\Lambda_0\Lotimes_\Lambda R\Gamma_{\Lambda_{>0}}(\Lambda)^\vee$ as the complex of $(\Lambda,\Lambda_0)$-bimodules, and the third line follows by replacing $\Lambda$ as a right $\Lambda$-module by an injective resolution. It follows that as left $\Lambda_0$-modules
\[
\Lambda_0\Lotimes_\Lambda R\Gamma_{\Lambda_{>0}}(\Lambda)^\vee=\RHom_{\Lambda^\circ}(\Lambda_0,\Lambda)^\vee,
\] 
which implies that $\Lambda_0\Lotimes_\Lambda H^i(R\Gamma_{\Lambda_{>0}}(\Lambda)^\vee)$ indeed has finitely generated cohomology as $\Lambda_0$-modules. 

The isomorphism $D_\Lambda\cong R\Gamma_{\Lambda_{>0}}(\Lambda)^\vee$ 
is a consequence of the uniqueness of  ``rigid'' dualizing complexes \cite[Definition 6.1, Proposition 8.2(1)]{van1997existence}. The fact that the right-hand side is rigid follows as in the proof of \cite[Proposition 8.2(2)]{van1997existence}, as for $D_\Lambda$ this follows from the proof of \cite[Proposition 5.7]{Yekutieli}. 
\end{proof}

\begin{lemma}\label{lem:fingen}
Let $\Lambda$ be a left noetherian $\NN$-graded ring. Assume that $M$ is a right bounded complex of graded left $\Lambda$-modules with left bounded cohomology. 
Then the cohomology modules of $M$ are finitely generated $\Lambda$-modules if and only if the cohomology modules of  $\Lambda_0\Lotimes_\Lambda M$ are finitely generated $\Lambda_0$-modules.
\end{lemma}

\begin{proof}
We concentrate on the nonobvious direction.
\begin{step}
Assume that $M\in\Gr(\Lambda)$ has left bounded grading. By the graded Nakayama lemma, $M$ is finitely generated if and only if $\Lambda_0\otimes_\Lambda M$ is finitely generated (see e.g. \cite[Proposition 2.2]{ATVdB}). 
\end{step}

\begin{step}
Let now $M$ be as in the statement of the lemma and assume that the cohomology modules of $\Lambda_0\Lotimes_\Lambda M$ are finitely generated $\Lambda_0$-modules. 
Let $m$ be maximal such that $H^m(M)\neq 0$. Then $H^m(\Lambda_0\Lotimes_\Lambda M)=\Lambda_0\otimes_\Lambda H^m(M)$. This follows by the appropriate hypercohomology spectral sequence. Hence by Step 1, $H^m(M)$ is finitely generated. 
\end{step}

\begin{step}
Tensoring the distinguished triangle 
\[
\tau_{\leq m-1} M\to M\to H^m(M)[-m]\to 
\]
with $\Lambda_0$ yields the distinguished triangle
\[
\Lambda_0\Lotimes_\Lambda\tau_{\leq m-1} M\to \Lambda_0\Lotimes_\Lambda M\to \Lambda_0\Lotimes_\Lambda H^m(M)[-m]\to.
\]
It now follows by Step 2 that $\Lambda_0\Lotimes_\Lambda\tau_{\leq m-1} M$ has finitely generated cohomology. Now we repeat Steps 2,3 with $\tau_{\leq m-1}M$ replacing $M$.\qedhere
\end{step}
\end{proof}

\begin{remark}\label{rem:dualhh}
If $\Lambda$ is homologically homogeneous (c.f. Definition \ref{def:hh}) of dimension $d$ then $D_\Lambda=\omega_\Lambda[d]$, where $\omega_\Lambda:=\Hom_R(\Lambda,\omega_R)$, by \cite[Proposition 2.9]{stafford2008}.  
\end{remark}


\providecommand{\bysame}{\leavevmode\hbox to3em{\hrulefill}\thinspace}
\providecommand{\MR}{\relax\ifhmode\unskip\space\fi MR }
\providecommand{\MRhref}[2]{%
  \href{http://www.ams.org/mathscinet-getitem?mr=#1}{#2}
}
\providecommand{\href}[2]{#2}

\end{document}